\documentclass[a4paper]{amsart}
\usepackage{amssymb,latexsym}
\usepackage[utf8x]{inputenc}
\usepackage{amsmath,amsthm}
\usepackage{amsfonts,mathrsfs,url}
\usepackage{enumerate,units}
\usepackage[all]{xy}
\usepackage{graphicx}

\theoremstyle{plain}
\newtheorem{theorem}{Theorem}[section]
\newtheorem{corollary}[theorem]{Corollary}
\newtheorem{lemma}[theorem]{Lemma}
\newtheorem{proposition}[theorem]{Proposition}
\newtheorem{conjecture}[theorem]{Conjecture}

\newtheorem{fact}[theorem]{Fact}
\newtheorem*{claim}{Claim}
\newtheorem{nclaim}{Claim}
\newtheorem*{theorem*}{Theorem}

\theoremstyle{definition}
\newtheorem{definition}[theorem]{Definition}
\newtheorem{example}[theorem]{Example}

\theoremstyle{remark}
\newtheorem*{remark}{Remark}

\numberwithin{equation}{section}
\newcommand{\forkindep}[1][]{%
  \mathrel{
    \mathop{
      \vcenter{
        \hbox{\oalign{\noalign{\kern-.3ex}\hfil$\vert$\hfil\cr
              \noalign{\kern-.7ex}
              $\smile$\cr\noalign{\kern-.3ex}}}
      }
    }\displaylimits_{#1}
  }
}

\newenvironment{claimproof}[1][\proofname]
  {%
    \proof[#1]%
  }
  {%
    \endproof%
  }

\newcounter{step}                   
    {\hfill $\clubsuit$             
     \vspace{7pt}\par}
     
\newcommand{\CL}{{\mathcal L}}
\newcommand{\CK}{{\mathcal K}}

\newcommand{\CO}{{\mathcal O}}

\DeclareMathOperator{\aut}{Aut}
\newcommand{\per}[1]{#1^{1/p^{\infty}}}
\newcommand{\chr}[1][K]{\mathrm{char}(#1)}
\DeclareMathOperator{\Spec}{Spec}
\date{\today}
\author[Y. Halevi]{Yatir Halevi$^*$}
\thanks{$^*$Partially supported by the European Research Council grant 338821, by ISF grant No. 181/16 and ISF grant No. 1382/15.}
\address{$^*$Department of mathematics\\
	Ben Gurion University of the Negev\\
	Be'er Sehva\\
	Israel}
\email{yatirbe@post.bgu.ac.il}
\urladdr{http://ma.huji.ac.il/\textasciitilde yatirh/}

\author[A. Hasson]{Assaf Hasson$^\dagger$}
\thanks{$^\dagger$ Supported by ISF grant No. 181/16}
\address{$^\dagger$Department of mathematics\\
	Ben Gurion University of the Negev\\
	Be'er Sehva\\
	Israel} \email{hassonas@math.bgu.ac.il} \urladdr{http://www.math.bgu.ac.il/\textasciitilde hasson/}

\author[F. Jahnke]{Franziska Jahnke$^\ddagger$}
\address{$^\ddagger$ Westf\"{a}lische Wilhelms-Universit\"{a}t M\"unster \\
	Institut f\"ur Mathematische Logik und Grundlagenforschung \\
	Einsteinstr. 62, 48149 M\"unster\\ 
	Germany}\email{franziska.jahnke@uni-muenster.de}
\thanks{$^\ddagger$ Funded by the Deutsche Forschungsgemeinschaft (DFG, German Research Foundation) under Germany's Excellence Strategy
EXC 2044--390685587, Mathematics M\"unster: Dynamics--Geometry--Structure and the CRC 878}

\begin{document}
\title{Definable V-topologies, Henselianity and NIP}
\begin{abstract}
We initiate the study of definable V-topolgies and show that there is at most one such V-topology on a t-henselian NIP field.  Equivalently, we show that if $(K,v_1,v_2)$ is a bi-valued NIP field with $v_1$ henselian (resp. t-henselian) then $v_1$ and $v_2$ are comparable (resp. dependent). 

\noindent As a consequence Shelah's conjecture for NIP fields implies the henselianity conjecture for NIP fields. Furthermore, the latter conjecture is proved for any field admitting a henselian valuation with a dp-minimal residue field.  

\noindent We conclude by showing that Shelah's conjecture is equivalent to the statement that any NIP field not contained in the algebraic closure of a finite field is $t$-henselian. 
\end{abstract}

\maketitle

\section{Introduction}
The class of first order theories without the independence property (NIP, for short) is arguably the most important class of theories studied by model theorists at the present time. We study the model theory of NIP fields admitting a henselian valuation (i.e., admitting a non-trivial henselian valuation). This class of fields covers many natural theories of fields, e.g. separably closed fields, real closed fields, $p$-adically closed fields and Hahn fields over any such field. Our work is motivated by the conjecture, attributed to Shelah\footnote{Shelah stated a closely related conjecture for strongly NIP fields. The conjecture stated here is, apparently, folklore.}, that every infinite NIP field is either separably closed, real closed or henselian.

Shelah's conjecture, if true, has far reaching consequences (see \cite{HaHaJa} and references therein for a detailed discussion). Among others, it implies that a wide class of NIP fields admits quantifier elimination in a reasonable language, and that their theories are decidable and explicitly axiomatisable in the spirit of Ax-Kochen \cite{axkochen} and Er\v{s}ov \cite{ersov}. In fact, by an unpublished work of Anscombe and the third author \cite{AnJa}, Shelah's conjecture implies that up to elementary equivalence, the above list of NIP fields is almost exhaustive.

In its full generality Shelah's conjecture is considered our of reach of present techniques. The only instance of the conjecture known is due to Johnson, \cite{johnsonp}, in the special case of NIP fields of dp-rank one (i.e. dp-minimal fields). Johnson's proof brings forth the importance of understanding the definable valuations in NIP fields and the topology they induce on the field. In fact, one of the keys to his proof, which was also proven independently by Jahnke-Simon-Walsberg in \cite{JSW}, is showing that in a dp-minimal valued field, which is not algebraically closed, the topology induced by the valuation is the same as the one induced by any henselian valuation.

Recall that a V-topology on a field is a topology which is induced by either a valuation or an archimedean absolute value. Such a topology is \emph{definable} if there exists a bounded (with repsect to the valuation or absolute value) definable neighbourhood of $0$. Such a set generates the V-topology (see Section \ref{ss:def-V}). A V-topology is t-henselian if it satisfies an appropriate variant of the inverse function theorem, see Section \ref{ss:V-topAG}. 

With this terminology in mind, Johnson's key result from above translates to the fact that every dp-mininal field, which is not algebraically closed, admits a canonical V-topology coinciding with the unique t-henselian topology (i.e. the topology induced by any henselian valuation on the field). The topology constructed by Johnson is, assuming the field is not real closed nor algebraically closed, canonical in the sense that it is the unique definable V-topology on the field.

The main result of the present paper is that the same is true for all t-henselian NIP fields:
\begin{theorem*}[Theorem \ref{T:indep-val=ip}]
Let $K$ be a t-henselian NIP field, possibly with additional structure), in some language expanding the language of rings. Then $K$ admits at most one definable V-topology.
\end{theorem*}

In particular if $(K,v_1,v_2)$ is a bi-valued NIP field with $v_1$ t-henselian then $v_1$ and $v_2$ induce the same topology (Corollary \ref{C:hens-incompIP}). On a field which is not separably closed, and assuming $\aleph_0$-saturation, this t-henselian topology is induced by an externally definable valuation (Proposition \ref{P:defV=extdef}). If, in addition, it is not real closed, it is always induced by a definable valuation (without any saturation assumption), see \cite[Theorem 5.2]{JahKoe} or \cite[Theorem 6.19]{dupont}.

Thus, our main result shows, in particular, that Shelah's conjecture implies the existence of a canonical V-topology on any infinite NIP field which is not separably closed. One consequence of the above is that for a real closed field $R$, NIP isolates precisely the henselian valuations on $R$. Namely, for any valuation $v$ on $R$, $(R,v)$  has NIP if and only if $v$ is henselian if and only if the valuation ring corresponding to $v$ is convex with respect to the order (Proposition \ref{P:nip-val-rcf}).

The generalization of the last statement to arbitrary NIP fields is known as the \emph{henselianity conjecture}. Explicitly it states that any NIP \emph{valued field} is henselian (i.e., that if $(K,v)$ has NIP as a valued field then $v$ is henselian). We use the main result to show that if $K$ is a NIP field admitting a henselian valuation $v$ with dp-minimal residue field then $K$ satisfies the henselianity conjecture (Corollary \ref{C:hen-con-dpmin-res}). In particular, any Hahn field over such fields satisfy the henselianity conjecture. We settle the following heuristic common belief.

\begin{theorem*}[Proposition \ref{P:shelah-implies-hens}]
Shelah's conjecture for NIP fields implies the henselianity conjecture.
\end{theorem*}

The proof of the main result generalizes methods used by Johson \cite[Chapter 11]{johnson}, Montenegro \cite{montenegro} and Duret \cite{duret} by proving the following strong approximation theorem on curves.

\begin{theorem*}[Proposition \ref{P:approx-for-curves}]
Let $K$ be a perfect field and $\tau_1,\tau_2$ two distinct V-topologies on $K$ with $\tau_1$ t-henselian. Let $C$ be a geometrically integral separated curve over $K$, $X\in \tau_1$ and 
$Y\in \tau_2$ 
infinite subsets of $C(K)$. Then $X\cap Y$ is nonempty.
\end{theorem*}

The topological theme is recapitulated in the the final section. We show that a well behaved interaction between definable sets and open sets is equivalent, without further model theoretic assumptions, to $t$-henselianity. This allows us to deduce the following new (to the best of our knowledge) general characterization of henselianity in terms of $t$-henselianity.

\begin{theorem*}[Theorem \ref{T:t-h}, Proposition \ref{P:t-hen-algext}]
Let $(K,v)$ be a valued field. Then $(K,v)$ is henselian if and only if $(Kv_\Delta,\bar v)$ is t-henselian for every coarsening $v_\Delta$ of $v$.

Moreover, if $K$ is not real closed then $v$ has a non-trivial henselian coarsening if and only if for every algebraic extension $L/K$, any extension of $v$ to $L$ is t-henselian.
\end{theorem*}

As a final result, we use all of the above to show the following reformulations of Shelah's Conjecture.

\begin{theorem*}[Corollary \ref{C:reform-shelah}]
The following are equivalent:
\begin{enumerate}
\item Every infinite NIP field is either separably closed, real closed or admits a non-trivial henselian valuation.
\item Every infinite NIP field is either separably closed, real closed or admits a non-trivial definable henselian valuation.
\item Every infinite NIP field is either t-henselian or the algebraic closure of a finite field.
\item Every infinite NIP field is elementary equivalent to a henselian field.
\end{enumerate}
\end{theorem*}

\subsubsection*{Acknowledgments}
We would like to thank Vincenzo Mantova for suggesting the present formulation of Proposition \ref{P:approx-for-curves}. Furthermore, we would like to thank Itay Kaplan for allowing us to use Lemma \ref{L:itay}. 
\section{Preliminaries}

\subsection{Valued Fields}
We will be using standard valuation theoretic terminology and notation. For ease of reference, we list below the main definitions and notations. 
See \cite{EnPr} for more background on valued fields. 

A valued field $(K,v)$ is a field $K$ together with a group homomorphism $v:K^\times\to \Gamma$ satisfying $v(x+y)\geq \min\{v(x),v(y)\}$, where $\Gamma$ is an ordered abelian group. We may extend $v$ to $K$ by defining $v(0):=\infty$. The set $\mathcal{O}_v:=\{x:v(x)\geq 0\}$ is called the valuation ring and $\mathcal{M}_v:=\{x:v(x)>0\}$ is its unique maximal ideal. The field $Kv:=\mathcal{O}_v/\mathcal{M}_v$ is called the residue field and the canonical projection $res:\mathcal{O}_v\to Kv$ is called the residue map. The value group $\Gamma$ is also denoted by $vK$. Note that $vK\simeq K^\times/\mathcal{O}_v^\times$ (as ordered groups) implying that giving a valuation $v$ on $K$ is 
(up to an isomorphism of the value group) 
the same as specifying its valuation ring.

A valuation $w:K^\times\to \Gamma'$ is a \emph{coarsening} of $v$ if $\mathcal{O}_v\subseteq \mathcal{O}_w$. There is a natural bijection between valuations coarsening $v$ and convex subgroups of the valuation group of $v$ (see \cite[Lemma 2.3.1]{EnPr}).
Two valuation $v_1$ and $v_2$ on $K$ are \emph{comparable} if $\mathcal{O}_{v_1}\subseteq \mathcal{O}_{v_2}$ or $\mathcal{O}_{v_2}\subseteq \mathcal{O}_{v_1}$, they are \emph{dependent} if their minimal common coarsening $\mathcal{O}_{v_1}\cdot \mathcal{O}_{v_2}$ is not equal to $K$. The topology induced by a valuation $v$ on a field $K$ is the field topology given by the neighbourhood base 
$\{x \in K : v(x)>\gamma \}_{\gamma\in vK}$ of open balls
of radius $\gamma$ around zero. More generally, we write
$B_\gamma(b,K)=\{x \in K : v(x-b) > \gamma\}$ for the open ball of radius 
$\gamma$ around $b$ in a valued field $(K,v)$. We write $B_\gamma(b)$ if 
the underlying field is obvious from the context, and $B_\gamma(K)$ in case
$b=0$.
%

\subsection{V-topological fields}\label{ss:V-topAG}
Recall that a $V$-topological field $(K,\tau)$ is a non-discrete Hausdorff topological field satisfying that for every $W\in \tau$ there exists $U\in \tau$ such that $xy\in U$ implies that either $x\in W$ or $y\in W$, see, e.g., \cite[Appendix B]{EnPr}. By a theorem of D\"urbaum and Kowalsky (see \cite[Theorem B.1]{EnPr}), a field supports a V-topology if and only if there exists either an (archimedean) absolute value or a (non-archimedean) valuation on $K$ inducing this topology. Every ordered field is also a V-topological field: if the order is archimedean the absolute value inherited from its embedding into the reals gives the desired topology. Otherwise, the natural valuation (which has the convex hull of 
$\mathbb Z$ as its valuation ring) is non-trivial. 

\begin{definition}\label{D:ift}
The \emph{V-closure} of a V-topological field $(K,\tau)$ is the completion of $(K,\tau)$ if $\tau$ is induced by an absolute value, and its henselization if 
$\tau$ is induced by a valuation.
A valued field $(K,\tau)$ is \emph{V-closed} if it is equal to its closure.
\end{definition}

We view a valued field as a substructure (in the language of valued fields) of its completion. In general, open sets are not $\bigvee$-definable, but nevertheless they are still a union of definable open balls, and thus any such open set $U$ in $K$ can also be interpreted in its $V$-closure, $\widehat K$, as a union of definable open balls. Because the completion is an immediate extension and using the ultra-metric triangle inequality, this goes through also in the other direction, as explained in the following simple lemma,  allowing us -- for certain arguments of an analytic nature -- to move back and forth between a field and its $V$-closure: 

\begin{lemma}\label{L:ift-"denseness"}
Let $(K,\tau)$ be a V-topological field with V-closure $(\widehat K,\widehat\tau)$. Then for every $\widehat W\in \widehat\tau$, $\widehat W\cap K$ is a $\tau$-open set and $\widehat{\widehat W\cap K}$, the set $\widehat W\cap K$ defines in $\widehat K$, is a $\widehat \tau$-open subset of $\widehat W$.
\end{lemma}
\begin{proof}
Let \[\widehat W=\bigcup_i \widehat B_{\gamma _i}(b_i),\] 
where $\widehat B_{\gamma_i}(b_i)\subseteq \widehat K$ is the ball defined by $|x-b_i|<\gamma_i$ if $\tau$ is induced by an absolute value and 
$v(x-b_i)>\gamma_i$ if $\tau$ is induced by a valuation. In the former case,
we have 
\[
\widehat W\cap K=\bigcup_i (\widehat B_{\gamma _i}(b_i)\cap K)=\bigcup_{a\in \widehat W\cap K} B_{\delta_i(a)}(a)
\]
where $B_{\delta_i(a)}(a)\subseteq K$ is the ball $|x-a|<\gamma_i-|a-b_i|$. In the latter case, by the ultra-metric inequality
\[
\widehat W\cap K=\bigcup_i (\widehat B_{\gamma _i}(b_i)\cap K)=\bigcup_{a\in \widehat W\cap K} B_{\gamma_i}(a).
\]
The rest is clear. 
%

\end{proof}

Following Prestel and Ziegler, a V-topological field is 
topologically henselian (t-henselian) if 
it satisfies a topological version of Hensel's lemma. This is made precise in the next definition. Here, for $K$ a field and $U \subseteq K$, we write
$U[X]^m$ to denote the set of polynomials with coefficients in $U$ of degree
at most $m$.
\begin{definition}
Let $(K,\tau)$ be a V-topological field. 
We say that $(K, \tau)$ is t-henselian, if for every $n \geq 1$ there is $U \in \tau$ such that every polynomial $f \in X^{n+1}+X^n+U[X]^{n-1}$ has a zero 
in $K$.
\end{definition}

Just like with Hensel's lemma for valued fields, there are many equivalent ways
of defining t-henselianity. The following characterization is the one which
will be of most use to us:
\begin{fact}[The Implicit Function Theorem]\cite[Section 4]{FVKimplicit}\label{F:implicit}
Let $(K,\tau)$ be a V-topological field.  $(K,\tau)$ is t-henselian if the implicit function theorem is true in $(K,\tau)$ in the following form:\\

\noindent Let $f: K^{n+m}\to K^m$ be a polynomial map. Assume that $(\bar a,\bar b)\in K^{n+m}$ satisfies $f(\bar a,\bar b)=0$ and that the matrix $\left( \partial f_i/\partial y_j (\bar a,\bar b)\right)_{1\leq i,j\leq m}$ is invertible.  Then there exist $\tau$-open neighbourhoods $0\in U\subseteq K^n$ and $0\in V\subseteq K^m$, such that for all $\bar a'\in \bar a + U$ there is exactly one solution $\bar b'\in \bar b+V$ for $f(\bar a',\bar b')=0$. Moreover the map $\bar a'\mapsto \bar b'$ from $\bar a+U$ to $K^m$ is continuous.
\end{fact}

\begin{remark}
By \cite[Theorem 7.4]{PrZ1978} $t$-henselianity of $(K,\tau)$ is already implied by the special case of the implicit function theorem where $m=1$. 
\end{remark}

Examples of t-henselian V-topological fields are V-closures of V-topological fields \cite[Section 7]{PrZ1978}. \\
%

%

The implicit function theorem is formally equivalent to the inverse function theorem, which will be more useful for our needs: 

\begin{lemma}[The Inverse Function Theorem]\label{L:PZ}
Let $(K,\tau)$ be a V-topological field. Then $(K,\tau)$ is t-henselian if and only if the following holds: 
for any polynomial map $r(\bar x): K^n\to K^n$ and $\bar d\in K^n$ such that  
\[|J_r(\bar d)|:=\mathrm{det}\left (\left( \partial r_i/\partial x_j (\bar d)\right)_{1\leq i,j\leq n}\right )\neq 0\]
  there are $\tau$-open subsets $U,V\subseteq K^n$ with $(\bar d,r(\bar d))\in U\times V$  and a continuous map $g:V\to U$ such that $(r\restriction U)\circ g=\textrm{Id}_V$ and $g\circ (r\restriction U)=\textrm{Id}_U$.
\end{lemma}
\begin{proof}
We use Fact \ref{F:implicit}. The inverse function theorem is a standard, formal, consequence of the the implicit function theorem by considering $f(\bar x,\bar y)=r(\bar y)-\bar x$. For the other direction, the implicit function theorem is a formal consequence of the inverse function theorem by considering $r(\bar x ,\bar y)=(\bar x,f(\bar x,\bar y))$. 
\end{proof}

The open mapping theorem is also a formal consequence of the inverse function theorem. For applications, a seemingly weaker version is handy: 
\begin{proposition}\label{P:etale=non-empty interior}
Let $(K,\tau)$ be a V-topological field. If for all open $U\subseteq K^n$, polynomial maps $f:U\to K^n$, and $\bar a\in U$ with $|J_f(a)|\neq 0$ there exists an open $f(\bar a)\in W\subseteq  f(U)$ then $(K,\tau)$ is t-henselian.
\end{proposition}
\begin{proof}
We use Lemma \ref{L:PZ}. Let $f:K^n\to K^n$ be a polynomial map and $\bar a \in K^n$ such that $|J_f(\bar a)|\neq 0$. Let $U_0\subseteq K^n$ be the $\tau$-open subset of points where the Jacobian determinant is non-zero and let $\widehat U_0$ be the set it defines in $(\widehat K,\widehat \tau)$, the V-closure of $(K,\tau)$. Note the following crucial observation:

\begin{claim}
For any $W\subseteq U_0$ , if $W$ is open so is $f(W)$.
\end{claim}
\begin{claimproof}
By assumption, for every $\bar b\in W$ there is a $\tau$-open subset of $f(W)$ containing $f(a)$. This implies that $f(W)$ is open.
\end{claimproof}

Applying one direction of Lemma \ref{L:PZ} to the henselian field $(\widehat{K}, \widehat{\tau})$ there exists an open subset $\bar a\in \widehat U\subseteq \widehat U_0 \subseteq \widehat K^n$  such that $f\restriction \widehat U$ is a homeomorphism onto $\widehat V:=f(\widehat U)$. So $f\restriction U$ is injective, open -- by the above claim -- and continuous (being a polynomial map), where $U:=\widehat{U}\cap K$. So it is a homeomorphism between $U$ and $f(U)$, and the result follows by applying the
other direction of Lemma \ref{L:PZ} to $(K,\tau)$.

\end{proof}

\subsection{Some standard algebraic geometry over perfect fields}

For the sake of completeness, we present the main algebro-geometric notions applied in the sequel.  The results of this subsection are certainly not new. The algebraic geometry needed here can be found in \cite{Lorenzini} and \cite{gortz}.

Many of the results below can probably be generalized to a non-perfect setting. One reason for favouring perfect fields is that the geometric notion of a variety over a given field coincides, when the field is perfect, with definability of the variety over that field in the model theoretic sense. 

By a variety over a perfect field $K$ we mean a geometrically integral scheme of finite type over $K$. When $V$ is affine there are finitely many polynomials over $K$ such that for any extension $L/K$ the set of $L$-rational points of $V$ is the zero-set of these polynomials. If $V$ is not affine, by elimination of imaginaries, the set of $K^{alg}$-rational points of $V$ is definable over $K$, see e.g., \cite[Section 7.4]{marker}. 

We note also that as V-topologies are field topologies, any $V$-topology on the field can be extended to a topology on any algebraic variety over the field by gluing affine patches, see, e.g., \cite[Proposition 2.3]{HiKaRi}.


There are several equivalent definitions of a morphism between varieties being \emph{\'etale} at a point, see \cite[Tag 02GU]{stacks-project}.  For ease of presentation we give the following definition, which will be most convenient for our needs:

\begin{definition}\cite[Tag 02GU, $(8)$]{stacks-project}\label{D:etale}
Let $X$ and $Y$ be varieties over a perfect field $K$. We say that $f:X\to Y$ is \emph{\'etale at $a\in X$} if there exist affine open neighbourhoods $U=\Spec(A)$ of $a$ and $V=\Spec(R)$ of $f(a)$, such that there exists a presentation 
\[A=R[x_1,\dots,x_n]/(f_1,\dots, f_n)\]
with 
$\det (\partial f_j/\partial x_i)_{1\leq i,j\leq n}$ mapping to an element of $A$ not in the prime ideal corresponding to $a$.
%
\end{definition}

Recall that if $X=\Spec\left(K[x_1,\dots,x_n]/(f_1,\dots,f_m)\right)$ is an affine variety over $K$, then the set of $K$-rational points, denoted by $X(K)$, is the set of morphisms $\mathrm{Hom}_K (\Spec (K),X)$. This set may be identified with the set of tuples $\bar a=(a_1,\dots,a_n)\in K^n$ satisfying $f_1(\bar a)=\dots=f_m(\bar a)=0$ and thus corresponding to the prime ideal $\mathfrak{q}_{\bar a}\triangleleft K[x_1,\dots,x_n]/(f_1,\dots,f_m)$ of all polynomials vanishing on $\bar a$. If $X$ is a variety over $K$ then a $K$-rational point of $X$ is a $K$-rational point of one of its open affine subsetes. We may thus use the following simplification of the definition of an \'etale morhpism at $K$-rational points:

Let $X$ and $Y$ be varieties over a perfect field $K$. Then $f:X\to Y$ is \emph{\'etale at $a\in X(K)$} if there exist affine open neighbourhoods $U=\Spec\left(K[x_1,\dots, x_{k+n}]/(g_1,\dots,g_l,f_1,\dots,f_n)\right)$ of $\mathfrak{q}_a$ (the prime ideal corresponding to $a$) and $V=\Spec\left(K[x_1,\dots,x_k]/(g_1,\dots,g_l)\right)$ of $f(\mathfrak{q}_a)$, with \[g_1,\dots,g_l\in K[x_1,\dots, x_k],\, f_1,\dots,f_n\in K[x_1,\dots,x_{k+n}]\] such that $\det (\partial f_j/\partial x_i)_{1\leq j\leq n,k+1\leq i\leq k+n}(a)\neq 0$ and the morphism $f$ restricted to $U$ is the projection onto the first $k$-coordinates.

We recall that if $f_1,\dots,f_m\in K[x_1,\dots,x_n]$ then $V(f_1,\dots,f_n)\subseteq K^n$ is the Zariski closed subset of $K^n$ defined by these polynomials.

\begin{proposition}\label{P:etale-in-t-hen}
Let $(K,\tau)$ be a t-henselian V-topological field and assume that $K$ is perfect. Let $h:X\to Y$ be a morphism between two varieties over $K$. If $h$ is \'etale at $p\in X(K)$, then there exist a $\tau$-open subset $p\in U'\subseteq X(K)$ such that $h\restriction U'$ is a $\tau$-homeomorphism onto $h(U')\subseteq Y(K)$.
\end{proposition}
\begin{proof}

It will suffice to prove the proposition for any affine open subvariety of $X$ containing $p$. Applying the above simplification, we may assume that 
\begin{enumerate}
\item $X$ is given by $V(g_1,\dots,g_l,f_1,\dots,f_n)\subseteq K^{k+n}$, where $g_1,\dots,g_l\in K[X_1,\dots X_k]$ and $f_1,\dots,f_n \in K[X_1,\dots ,X_{k+n}]$;
\item $\det(\partial f_j/\partial X_i)_{k+1\leq i,j\leq k+n}(p)\neq 0$;
\item $Y$ is given by $V(g_1,\dots, g_l)\subseteq K^k$ and
\item $h$ is the projection on the first $k$-coordinates.
\end{enumerate}

Consider $r:K^{k+n}\to K^n$ given by $r=(f_1,\dots,f_n)$. Note that $r(p)=0$. By Fact \ref{F:implicit}, there exists a $\tau$-open subset $W\subseteq K^k$ and a $\tau$-continuous function $s:W\to K^n$ such that $r(x,s(x))=0$. Consequently, $h\restriction (W\times K^n)\cap V(g_1,\dots,g_l,f_1,\dots, f_n)$ is a $\tau$-homeomorphism onto $W\cap V(g_1,\dots g_l)$, where the inverse function is $x\mapsto (x,s(x))$.

\end{proof}

The correspondence sending an algebraic curve $C$ over $K$ to its field of regular functions is an equivalence of categories between  geometrically integral nonsingular projective curves over a perfect field $K$ (with dominant rational maps as morphisms) and the category  of regular finitely generated field extensions of transcendence degree $1$. See \cite[Sections V.10, VII]{Lorenzini} for more details. 


Let $C$ be a geometrically integral nonsingular projective curve over a perfect field $K$ and $L/K$ a finite algebraic extension. The local ring $\mathcal O_{P,C}$ for $P\in C(L)$ is a discrete valuation ring of $K(C)$ (which when the context is clear we will denote by $\mathcal O_P$). It induces a valuation $v_P$ on $K(C)$, which is trivial on $K$, with residue field $L$.
An element $f\in\mathcal{O}_P$ is called a function on $C$ defined at $P$, the integer $v_P(f)$ is called the order of vanishing of $f$ at $P$, and if $v_P(f)>0$ then $f$ has a zero at $P$. If $f\in K(C)\setminus \mathcal{O}_P$ then $f$ has a pole of order $|v_P(f)|$ at $P$. 

Given a morphism $\varphi:X\to Y$ between two geometrically integral nonsingular projective curves over a perfect field $K$, its pullback induces a $K$-algebra homomorphism $\varphi^*:K(Y)\to K(X)$ with the obvious property that to $P\in X(L)$ corresponds a valuation $v_{f(P)}=v_p\circ \varphi^*$ on $K(Y)$. Note that $K(X)/\varphi^*(K(Y))$ is a finite algebraic extension.

\begin{proposition}\label{P:v-top-implicit-etale}
Let $(K,\tau)$ be a t-henselian V-topological field and assume that $K$ is perfect. Let $C$ be a geometrically integral nonsingular projective curve over $K$, $f\in K(C)^\times$ and $p\in C(K)$. If $v_p(f)=1$ then there is a $\tau$-open subset $p\in U\subseteq C(K)$ with $f\restriction U$ a $\tau$-homeomorphism.
\end{proposition}
\begin{proof}
By Proposition \ref{P:etale-in-t-hen}, the proof boils down to showing that $f$ is \'etale at $p$.
Since we could not find any good source for this fact we add the (easy) proof.

To any $f$ as above corresponds a non-constant morphism $f:C\to \mathbb{P}^1_K$ where $\mathbb{P}^1_K$ is the projective line over $K$. There is a corresponding $K$-algebra embedding of fields $f^*:K(\mathbb{P}^1_K)\to K(C)$. Assume that $K(C)\cong K(c)$ for some tuple $c$. Replacing $K(\mathbb{P}^1_K)$ with the isomorphic subfield of $K(C)$, we may assume we have a field extension $K(c)/K(f(c))$. 

By \cite[Tag 02GU $(6)\Leftrightarrow (8)$]{stacks-project}, the rational map $f$ is \'etale at $p$ if and only if
\begin{enumerate}
\item $\mathcal{O}_{f(p)}/\mathcal{M}_{f(p)}\hookrightarrow \mathcal{O}_p/\mathcal{M}_p$ is a finite separable field extension and
\item $\mathcal{M}_{f(p)}\mathcal{O}_p=\mathcal{M}_p.$
\end{enumerate}

Since both $p$ and $f(p)$ are $K$-rational points, $\mathcal{O}_{f(p)}/\mathcal{M}_{f(p)}= \mathcal{O}_p/\mathcal{M}_p=K$. For the second condition, since the valuation $v_p$ of $K(c)$ is an extension of $v_{f(p)}$ on $K(f(c))$, $f(c)$ is a uniformizer of $v_{f(p)}$ and the result follows.

\end{proof}

Let $C$ be as before. A \emph{divisor} on $C$ is an element of the free abelian group generated by the set $\{x_v:v\in \mathcal{V}(K(c)/K)\}$, where $\mathcal{V}(K(c)/K)$ is the set of discrete valuations on $K(C)$ which are trivial on $K$. So each element is a finite sum of elements of the form $a_vx_v$, where $a_v$ are integers. A divisor is called effective if $a_v\geq 0$ for all $v$. The degree of a divisor $D=\sum a_vx_v$ is defined to be $\deg(D):=\sum a_v\cdot [\mathcal{O}_v/\mathcal{M}_v:K]$. Each $f\in K(C)^\times$ corresponds to a divisor $\mathrm{div}(f)=\sum\limits_{v\in \mathcal{V}(K(C)/K)}v(f)x_v$, see \cite[Proposition VII.4.11]{Lorenzini}.

For each divisor $D$ of $C$ the set $H^0(D):=\{f\in K(C): \mathrm{div}(f)+D \text{ is effective}\}$ is a finite dimensional vector space over $K$ with dimension $h^0(D)$. The Riemann-Roch theorem implies that $h^0(D)\geq \mathrm{deg}(D)+1-g$, where $g$ is the genus of $C$, see \cite[Section IX.3]{Lorenzini}. As an application, we prove:

\begin{lemma}\label{L:divisor}
Let $C$ be a geometrically integral nonsingular projective curve over a perfect field $K$, and $X$ a subset of $C(K)$ of cardinality larger than the genus of $C$. Then there exists a non-constant rational function $f\in K(C)$ such that $f^{-1}(0)$ is a sum of distinct points in $X$ with no multiplicities.
\end{lemma}
\begin{proof}
Let $g$ be the genus of $C$ and let $p_1,\dots,p_{g+1}$ be distinct points in $X\subseteq C(K)$.

Let $D$ be the divisor $\sum_j p_j$ on the curve $C$. By Riemann-Roch $h^0(D)\geq \deg(D)+1-g=2$.

We can thus find a non-constant $g\in K(C)$, with $\mathrm{div}(g)+D$ effective, i.e. the divisor of poles of $g$ is a subset of $D$, so every pole has multiplicity $1$ and is in $X$. Take $f=1/g$.
\end{proof}

\begin{lemma}\label{L:U-infinite}
Let $(K,\tau)$ be a t-henselian V-topological field and assume that $K$ is perfect. Let $C$ be a geometrically integral nonsingular projective curve over $K$. Then every $\tau$-open non-empty subset $U\subseteq C(K)$ is infinite.
\end{lemma}
\begin{proof}
Since every henselian field is large and being large is an elementary property, see \cite{large}, $K$ is also large. Indeed, every t-henselian V-topological field is elementary equivalent to a henselian field by \cite[Theorem 7.2]{PrZ1978}. Thus $C(K)$ is infinite.

Consider the divisor $p+g\cdot q$, where $p\in U$, $p\neq q\in C(K)$ and $g$ is the genus of the curve. As in the proof of Lemma \ref{L:divisor}, there exists a non-constant rational function $f\in K(C)$ with $v_p(f)=1$. By Proposition \ref{P:v-top-implicit-etale}, and since every $\tau$-open subset of $K$ is infinite, $U$ is infinite.
\end{proof}

%

%

\subsection{Model Theory and NIP}\label{ss:NIP}
We will assume knowledge of some basic model theory, see e.g. \cite{TZ}  and \cite{Sim2015}. We recall some facts regarding NIP theories, i.e. first order theories that do not have the independence property. 
As a matter of notational convenience, below and throughout for a tuple $a$ by ``$a$ is in $M$'' or $a\in M$ we mean  $a\in M^{|a|}$.

\begin{definition}
Let $T$ be a first order theory in a language $\mathcal L$. A first order $\mathcal L$-formula $\varphi(x,y)$ has NIP (is dependent) with respect to $T$ if for some integer $k$, there are no model $M$ of $T$ and families of tuples  
$(a_i: i<k)$ and $(b_J: J\subseteq [k])$ in $M$ with
\[M\models \varphi(a_i,b_J) \Leftrightarrow i\in J.\]
The theory $T$ has NIP if every formula has NIP with respect to $T$. 
\end{definition}

A structure $\mathcal M$ is NIP if its theory is, and if $\mathcal M$ is NIP and $\mathcal N$ is interpretable in $\mathcal M$ then $\mathcal N$ is also NIP. This fact implies that, for example, a valued field has NIP in the three sorted language (with the base field, the value group and the residue field with the obvious maps between them) if and only if it has NIP in the language of rings expanded by a unary predicate for the valuation ring.  
%

Let $M$ be a first order structure, and $M\prec N$ an $|M|^+$-saturated elementary extension. The \emph{Shelah expansion of $M$}, denoted  $M^{sh}$, is $M$ augmented by a predicate for each externally definable subset of $M$ (i.e. subsets of the form $\varphi(M,c)$ for some $c\in N$). A priori, the Shelah expansion depends on $N$ but in reality all $|M|^+$-saturated expansions induce the same definable subsets, see \cite[Section 3.1.2]{Sim2015}.

\begin{fact}\cite[Proposition 3.23, Corollary 3.24]{Sim2015}
Let $M$ be a NIP structure. Then $M^{sh}$ admits quantifier elimination  and consequently, $M^{sh}$ has NIP as well.
\end{fact}

Note that if $(K,v)$ is a valued field then any coarsening $w$ of $v$ is externally definable, implying the following useful fact that  we will use implicitly throughout:

\begin{fact}
Let $K$ be a NIP field (possibly with additional structure) and $v$ a definable valuation.  Then for every coarsening $w$ of $v$, $(K,v,w)$ has NIP.
\end{fact}
\begin{proof}
Every coarsening of $v$ is uniquely determined by a convex subgroup of $vK$. As a convex subset it is an externally definable set and hence definable in the Shelah expansion of $K$. Thus adding a predicate for the coarsening preserves NIP.
\end{proof}

\begin{remark}
In what follows, we will refer to theories that are ``strongly dependent'' or of ``finite dp-rank''. We will not be using explicitly these notions and it is enough for our needs to note that these are sub-classes of NIP theories (see \cite[Chapter 4]{Sim2015}), which -- like NIP theories -- are stable under interpretations and expansions by externally definable sets. 
\end{remark}

\section{Definable V-topologies}\label{ss:def-V}

There are several languages for studying V-topological fields as first-order structures. Two natural choices for such languages are: adding a sort for the topology (as was done in \cite{PrZ1978}) and adding a symbol for a valuation or an absolute value inducing the V-topology.  In order not to commit ourselves to one such approach, we suggest the following common generalization: 

\begin{definition}
A field $K$ (possibly with additional structure) is \emph{elementarily V-topological} if it admits a uniformly definable 0-neighbourhood base for a V-topology. In this case, we also say that $K$ admits a \emph{definable V-topology}.
\end{definition} 


As we will see below, fairly little is needed for a V-topological field to be elementarily V-topological. For the next lemma we recall that a subset $A\subseteq K$ of a V-topological field is \emph{bounded} if it is contained in an open ball (with respect to a valuation or an absolute value inducing the topology, recall Section \ref{ss:V-topAG}). This definition does not depend on the choice of the valuation (see \cite[Appendix B]{EnPr}).  

\begin{lemma}\label{L:def-top}
The following are equivalent for a V-topological field $(K,\tau)$:
\begin{enumerate}
\item $K$ is elementarily V-topological;
\item $K$ has a definable bounded neighbourhood of $0$;
\item $K$ has a 0-neighbourhood base consisting of definable sets. 
\end{enumerate}

In particular, every $L\equiv K$ has a definable V-topology making it a V-topological field.
\end{lemma}
\begin{proof}
$(1)\implies (3)\implies (2)$ are obvious. For $(2)\implies (1)$ note that if $U$ is a definable bounded neighbourhood of $0$ then $\{cU:c\in K^\times\}$ is a uniformly definable $0$-neighbourhood base. 

As for the concluding part of the lemma, we apply (1): assume that $\phi(x,\bar y, \bar a)$ is a $K$-definable $0$-neighbourhood base for the topology. Then the statement ``there exists $\bar z$ such that $\phi(x,\bar y, \bar z)$ is a $0$-neighbourhood base for a V-topology'' is elementary. 
%
%
%
\end{proof}
In the arXiv preprint version of \cite{JSW}, V-topologies with
a definable bounded neighbourhood of $0$ were called \emph{definable type 
V topologies}.

The above lemma allows us to replace any given elementarily V-topological field with a sufficiently saturated extension which, in turn, allows us to assume that there is a valuation inducing the V-topology. Though essentially folklore, we add the proof for completeness.

\begin{lemma}\label{L:from absolute-to-valuation}
Every $\aleph_0$-saturated elementarily V-topological field admits a non-trivial non rank-one valuation inducing the V-topology.
\end{lemma}
\begin{proof}
Let $K$ be an $\aleph_0$-saturated elementary V-topological field. Assume there exists an absolute value $|\cdot|$ inducing the definable V-topology. Since the V-topology is definable, for any $n\in \mathbb{N}$, we may choose a definable set $U_n$ contained in $\{x\in K: |x|<1/n\}$. By $\aleph_0$-saturation there exists a non-zero element in $\bigcap_{n\in \mathbb{N}} U_n$, contradiction.

The proof for a rank-one valuation is similar.
\end{proof}

At this level of generality, we can not expect elementarily V-topological fields to admit a definable valuation, as witnessed by any non-archimedean real closed field. Below we show, however, that any sufficiently saturated such field admits an externally definable valuation.

The following argument, due to Kaplan, shows that, under reasonable assumptions,  $\bigvee$-definable sets are externally definable.

\begin{lemma}[Kaplan]\label{L:itay}
Let $M\prec N\prec \mathcal{U}$ be first order structures, in any language, with $\mathcal{U}$ a monster model and $M,N$ small. Let $\varphi(x,y)$ be a formula and $\{\varphi(x,a_i):a_i\in N, i<\alpha\}$ a small family of instances of $\varphi(x,y)$ satisfying $\varphi(M,a_i)\neq \emptyset$ for every $i<\alpha$. If for every $i<j<\alpha$, $\varphi(M,a_i)\subseteq (M,a_j)$ then there exists some $c\in \mathcal{U}$ such that 
\[\bigcup_{n<\alpha} \varphi(M,a_n)=\varphi(M,c).\]
\end{lemma}
\begin{proof}
Let $D$ be an ultrafilter on $N$ extending the filter base $\left\{ \{a_i:i>n\}: n<\alpha\right\}$.
By \cite[Example 2.17(1)]{Sim2015} there exists a global type $p_D(y)$, defined by
\[p_D\vdash \psi(b,y) \iff \psi(b,N)\in D.\]

Let $c\models p_D|N$. We claim that \[\bigcup_{n<\alpha} \varphi(M,a_n)=\varphi(M,c).\]
Indeed, if $d\in M$ such that $\varphi(d,c)$ holds, then, by the definition of $p_D$, $\varphi(d,N)\in D$. Consequently, by the definition of $D$, there exists $i<\alpha$ such that $a_i\in \varphi(d,N)$, as needed.

For the other direction, assume that $\varphi(d,a_n)$ holds for some $n<\alpha$. Since the family is increasing, $\{a_i: i\geq n\}\subseteq \varphi(d,N)$ so $\varphi(d,N)\in D$ and thus finally $\varphi(d,c)$.
\end{proof}

Using the above lemma we can prove, following closely \cite[Appendix B]{EnPr}: 


\begin{proposition}\label{P:defV=extdef}
Every $\aleph_0$-saturated elementarily V-topological field admits an externally definable valuation inducing the V-topology.
\end{proposition}
\begin{remark}
$\aleph_0$-saturation is only needed to assure the existence of a non rank-one valuation inducing the topology.
\end{remark}
\begin{proof}
Let $K$ be an $\aleph_0$-saturated elementarily V-topological field.
By Lemma \ref{L:from absolute-to-valuation} $K$ admits a non rank-one valuation $v$ inducing the definable V-topology. Consequently, for any $d\neq 0$ if $v(d)>0$ then $\{d^{-n}:n<\omega\}$ is a bounded set. 

Let $U$ be a definable bounded neighbourhood of $0$, hence there exists $\gamma\geq \alpha\in vK$ satisfying $B_\gamma\subseteq U\subseteq B_\alpha$, where $B_\gamma$ and $B_\alpha$ are, respectively, the open ball around $0$ of radius $\gamma$ and of radius $\alpha$ (with respect to $v$). By replacing $U$ with $U\cup (-U)$ we may assume that $U$ is closed under additive inverses. Furthermore, by replacing $U$ with $\{x\in K: xU\subseteq U\}$ we may additionally assume that $1\in U$ and $UU\subseteq U$ (see \cite[Lemma B.5]{EnPr}). In particular $1\in B_\alpha$ and thus $\alpha<0$, so we may assume that $\gamma>0>\alpha$ and that $\gamma +\alpha>0$.

Let $d\neq 0$ be an element satisfying $v(d)>\max\{2\gamma,\gamma-\alpha\}$ (in particular $d\in U$).  We claim that $(K\setminus U)^{-1}\subseteq d^{-1}U$. Indeed, if $x\in K\setminus U$ then $v(x)\leq \gamma$ and so $v(x^{-1})\geq -\gamma$ which gives 
\[v(dx^{-1})>2\gamma-\gamma=\gamma.\] Hence $dx^{-1}\in U$ as needed.

For every $n<\omega$, set \[O_n:=\{x\in K:xU\subseteq d^{-n}U\}\] and \[O:= \bigcup_{n<\omega} O_n.\]
Since $U$ is definable, we may apply Lemma \ref{L:itay} after checking that the family $\{O_n:n<\omega\}$ is increasing. For this it is enough to show that $dU\subseteq U$. Let $x\in U$, in particular $v(x)>\alpha$, and thus \[v(dx)>2\gamma+\alpha>\gamma\] so $dx\in B_\gamma\subseteq U$, as needed. Consequently, by Lemma \ref{L:itay}, $O$ is externally definable. 

Note that, since $UU\subseteq U$ we obtain $U\subseteq O_0\subseteq O$, implying
\[(K\setminus O)^{-1}\subseteq (K\setminus U)^{-1}\subseteq d^{-1}U\subseteq d^{-1}O\subseteq O.\]
Consequently, for every $x\in K$ either $x\in O$ or $x^{-1}\in O$. Since $0\in U$ we see that $0,1\in O$ and noting that if $a,b\in O_n$ then $ab\in O_{2n}$ we also get that $OO\subseteq O$.

We show that $O$ is closed under addition.
First observe that by the choice of $d$, \[d(U+U)= dU+dU\subseteq dB_\alpha+dB_\alpha\subseteq B_\gamma+B_\gamma=B_\gamma\subseteq U\] and thus $U+U\subseteq d^{-1}U$. Now, let $w_1+w_2\in O+O$ thus there exists $n<\omega$ with $w_1,w_2\in d^{-n}U$ and so \[w_1+w_2\in d^{-n}U+d^{-n}U\subseteq d^{-n}(U+U)\subseteq d^{-n-1}U\subseteq O.\]

Note that $O$ is closed under additive inverses since it is true for each of the $O_n$. Finally, since $O$ is bounded (recall $\{d^{-n}:n<\omega\}$ is bounded) it is a bounded neighbourhood of $0$ and hence defines the same V-topology as $v$. Combining everything together $O$ is an externally definable valuation ring defining the same topology as $v$, as needed.

\end{proof}
A version of this proposition 
(albeit with a different proof) is also contained
in the arXiv preprint version of \cite{JSW}.

Building on an argument of Johnson's (\cite[Lemma 9.4.8]{johnson}), we finish this section with a lemma that will not be used in the rest of this note. It allows extending definable V-topologies to finite extensions.

\begin{lemma}\label{L:finite-extension-def-V}
	Let $K$ be an elementarily $V$-topological field, $L\ge K$ a finite extension. Then $L$ with its induced structure from $K$ is elementarily $V$-topological, and the topology on $L$ can be chosen to induce the topology on $K$. 
\end{lemma}
\begin{proof}
	We need to show that there is, in $L$ (with the induced structure from $K$), a $V$-topology $\tau$ extending the topology on $K$ and a definable bounded $0$-neighbourhood $U\in \tau$. 	Since $L$ is interpretable in $K$, for every elementary extension, $\CK\succ K$ there is $\CL\succ L$ such that $(\CL, \CK)\equiv (L,K)$. By Lemma \ref{L:def-top}, $\CK$ is elementarily $V$-topological and if we show that $\CL$ is $V$-topological then so is $L$. Of course, if the topology on $\CL$ extends the topology on $\CK$ then, as $(\CL, \CK)\equiv (L,K)$ and all topologies concerned are definable, the same remains true in $(L,K)$. So we may assume that $K$ is sufficiently saturated and that the $V$-topology on $K$ is induced by a valuation $v$, and fix some extension $v'$ of $v$ to $L$.

	Next, we note that we may assume that $L$ is a normal extension of fields. Indeed, let $\tilde L\ge L\ge K$ be the normal closure of $L$. Because $\tilde L$ is a finite extension of $L$ it is interpretable in $L$, and if we find a topology in $\tilde L$ satisfying the requirement,
	(which is equivalent to what we want by Lemma \ref{L:def-top}), 
	so will its restriction to $L$. 
	
	Denote by $\mathcal{M}\subseteq\CO\subseteq K$ (respectively $\mathcal{M}'\subseteq\mathcal{O}'\subseteq L$) the maximal ideal and valuation ring of $v$ (respectively of $v'$). Let $a_1\dots, a_k\in \mathcal{O}'$ be such that for all $\sigma\in \aut(L/K)$ there exists some $a_i\in \CO'\setminus \sigma(\CO')$. We may assume that $v'(a_i)=0$. Indeed, replacing $a_i$ with $a_i^n$ for a suitable $n$ we may assume that $v'(a_i)\in vK$ (because $v'L/vK$ is torsion). So let $c_i\in K$ be such that $v'(c_i)=-v'(a_i)$ and it is immediate that $\{c_ia_i\}_{i=1}^k$ satisfies the requirements. Denote $\bar a:=(a_1,\dots, a_k)$.
	
	Fix some $K$-definable bounded $U\supseteq \mathcal{M}$, (e.g., 
$U=cW$ for any definable bounded $0$-neighbourhood $W$ and $c$ of suitable valuation). By \cite[Claim 9.4.9 and proof of Lemma 9.4.8]{johnson}  there exists a natural number $d$ such that $x\in \CO'$ if and only if there is no polynomial $P$ of degree at most $d$ with parameters in $\mathcal{M}$ such that $P(x,\bar a)=1$. 
	So the set $U_1$ of all $x$ not satisfying any equation $P(x,\bar a)=1$ for $P$ a polynomial of degree at most $d$ with parameters in $U$ is contained in $\CO'$, i.e. $U_1$ is a definable bounded set in $L$. 
	
	It remains to show that $U_1$ is a $0$-neighbourhood. Because $U$ is bounded there exists some $\gamma$, necessarily non-positive, such that $U\subseteq B_\gamma(K)$ (the open ball of radius $\gamma$ around $0$ in $K$). The set of $x\in L$ such that $P(x,\bar a)=1$ for some non-constant $P$ over $B_{-(d+1)\gamma}(K)$ of degree at most $d$ does not contain any element of $B_{\gamma}(L)$.  To see this, observe that if $x\in B_{\gamma}(L)$ then for some $1\leq n\leq d$, \[0=v'(1)=v'(P(x,\bar a))\geq -(d+1)\gamma+v'(x^n)> -(d+1)\gamma +n\gamma\geq 0,\] which is impossible. So the set $U_1$ contains the open ball $ B_{\gamma}(L)$.
\end{proof}

\section{Strong Approximation on Curves}
The arguments here generalise ideas used by Johnson in his thesis (\cite[Section 11.3]{johnson}) as part of the proof that certain models of valued fields are existentially closed. In fact, as we show below, his arguments give a specific case of a strong approximation result for curves. Though this can, probably, be extended further,  there is no hope of finding strong approximation result for all varieties over a V-topological field $(K, \tau)$. For example, if a curve has only finitely many $K$-points then -- because $V$-topologies are Hausdorff -- strong approximation must fail.

\begin{fact}[The Strong Approximation Theorem for V-topological Fields]\label{F:strong-approx-ift}\cite[Corollary 4.2]{PrZ1978} 
Let $K$ be a field with two independent (i.e. defining different topologies) V-topologies $\tau_1$ and $\tau_2$. Then for every non-empty $U_1\in \tau_1$ and $U_2\in \tau_2$ there exists $x\in K$ with $x\in U_1\cap U_2$.
\end{fact}
%

\begin{proposition}\label{P:approx-for-curves}
Let $K$ be a perfect field and $\tau_1,\tau_2$ two distinct V-topologies on $K$ with $\tau_1$ t-henselian. Let $C$ be a geometrically integral separated curve over $K$, $X\in \tau_1$ and 
$Y\in \tau_2$ 
 subsets of $C(K)$, both containing at least one non-singular point of $C(K)$. If $Y$ is infinite then $X\cap Y$ is nonempty.
\end{proposition}
\begin{proof}
Let $\tilde C$ be the non-singular projective curve who is birational isomorphic to $C$ and let $\tilde X$ and $\tilde Y$ be the images of $X$ and $Y$ under this birational isomorphism. By assumption, $\tilde X$ and $\tilde Y$ are non-empty. Furthermore, $\tilde X$ is infinite by Lemma \ref{L:U-infinite} and $\tilde Y$ is infinite by assumption. If we show that $\tilde X\cap \tilde Y$ is nonempty, by removing this point and repeating the argument, it will have to be infinite and thus $X\cap Y$ is non-empty. We may thus assume that $C$ is non-singular projective. 

Applying Lemma \ref{L:divisor}, 
we find a non-constant rational function $f\in K(C)$, $f:C\to K$, whose divisor of zeros has no multiplicities and consists of points $P_1,\dots,P_m$ in $X$. By, e.g., \cite[page 249]{Lorenzini}, $f$ is a map of degree at most $m$.

\begin{claim}
There is a $\tau_1$-open $0\in U\subseteq K$ such that for every $y\in U$, $f^{-1}(y)$ consists of $m$ distinct points in $X$.
\end{claim}
\begin{claimproof}
Since each $P_j$ has multiplicity one, by Proposition \ref{P:v-top-implicit-etale} for each $j$ there is a $\tau_1$-open neighbourhood $W_j\subseteq C(K)$ of $P_j$ such that $f$ induces a $\tau_1$-homeomorphism between $W_j$ and an open neighbourhood of zero. We may assume that $W_j\subseteq X$ and that $W_j\cap W_{j'}=\emptyset$ for $j\neq j'$. Hence $U=\bigcap_{j=1}^m f(W_j)$ is an open neighbourhood of $0$ in $K$, and if $y\in U$ then $f^{-1}(y)$ has at least one point in each $W_j$ but because the $W_j$ are distinct and $f$ has degree $m$, we are done.
\end{claimproof}
Now consider $\tau_2$. By \cite[Paragraph after Theorem 7.9]{PrZ1978}, 
$(K,\tau_2)$ is dense in $K^{sep}$ with respect to any extension of $\tau_2$ to $K^{sep}$, the separable closure of $K$. In particular, it is dense in its V-closure $(\widehat K, \widehat \tau_2)$.
Let $C(\widehat K)$ be the set of $\widehat K$-rational points of $C$. 
As in the proof of Lemma \ref{L:ift-"denseness"} and since $Y$ is a union of balls over $K$, we may consider $Y(\widehat K)$, the set $Y$ defines in $(\widehat K,\widehat \tau_2)$.


Since $(\widehat K,\widehat \tau_2)$ is t-henselian, by Lemma \ref{L:U-infinite} $Y(\widehat K)$ is infinite and hence Zariski dense. By \cite[Proposition 4.6]{etale}, there exists a Zariski open neighbourhood of, say, $P_1$ in which $f$ is \'etale and hence $Y(\widehat K)$ has an \'etale point of $f$.
By Proposition \ref{P:etale-in-t-hen}, there exists a $\widehat \tau_2$-open subset $\widehat W$ of $Y(\widehat K)$ on which $f$ is a homeomorphism. In particular $\widehat V:=f(\widehat W)$ is a $\widehat \tau_2$-open subset of $\widehat K$. Since $K$ is $\tau_2$-dense in $\widehat K$, $\widehat V\cap K$ is non empty. Consequently, by Lemma \ref{L:ift-"denseness"}, $\widehat V\cap K$ defines a $\tau_2$-open subset $V(K)$ of $K$.

By the strong approximation theorem (Fact \ref{F:strong-approx-ift}), there is $y\in V(K)\cap U$.  Since $y\in U$, by the claim, all $m$ pre-images $f^{-1}(y)$ are elements of $X\subseteq C(K)$. Because $y\in V(K)$, there is some $x\in \widehat W$ mapping to $y$, so we must have $x\in X\subseteq C(K)$. Thus $x\in X\cap Y$.

\end{proof}


%

\section{Independent Valuations on Valued Fields}

Our aim in this section is to generalize results by Johnson \cite[Chapter 11]{johnson} and Montenegro \cite{montenegro}, regarding the dependence of valuations in certain classes of NIP henselian fields. Both proofs use techniques from \cite{duret}. We generalize their results to t-henselian NIP fields.  The strong approximation result proved in the previous section will play an important role in the argument. 

As in \cite{montenegro}, \cite{johnson} and \cite{duret}, a crucial step is the observation that $q$-th roots can be separated by open sets for any $q>1$. For the rest of this section, let $K^q$ denote the set of $q$-th powers of $K$.

\begin{lemma}\label{L:correspondence-between-roots}
Let $K$ be a field containing a primitive $q$-th root of unity $\xi_q$ and let $b\in K^q$, where $q$ is a prime number.
Let $v$ be a valuation on $K$ satisfying
\begin{enumerate}
\item $q\neq \chr[Kv]$, or
\item $q=2$ and $\chr[K]\neq 2$.
\end{enumerate}
If $v(b- 1)>\gamma$ for some non-negative $\gamma\in vK$, then there is a one-to-one correspondence \[j:\{a\in K: a^q=b\}\to \{0,\dots, q-1\}\] such that $v(a-\xi_q^{j(a)})>\gamma/q$.
%
\end{lemma}
\begin{proof}
  Let $a\in K$ with $a^q=b$. Since $v(a^q-1)>\gamma$, 
  \[
  \sum_{j=0}^{q-1} v(a-\xi_q^j)>\gamma.
  \] 
  As a result, there exists some $j$ with $v(a-\xi_q^j)>\gamma/q$. If $v(a-\xi_q^{j})>\gamma/q$ and $v(a-\xi_q^{l})>\gamma/q$, then $\xi_q^{j},\xi_q^{l}\models v(x-a)>\gamma/q$ and consequently, $v(\xi_q^j-\xi_q^{l})>\gamma/q> 0$. Since $q\neq \mathrm{char}(Kv_i)$, the residue of $1$ has $q$ distinct $q$-th roots of unity, so $j=l$.

The map is onto since if $v(a-\xi_q^j)>\gamma/q$ then $v(a\xi_q^{\widehat j}-\xi_q^{j+\widehat j})>\gamma/q$ for any $0\leq \widehat j\leq q-1$.

%
%
%
\end{proof}

We also need the following result due to Johnson: \begin{fact}\cite[Lemma 9.4.8]{johnson}\label{F:def-of-val-finite-extension}
Let $K$ be a field with some structure, and $L/K$ a finite extension. Suppose that $\mathcal{O}$ is a definable valuation ring on $K$. Then each extension of $\mathcal{O}$ to $L$ is definable (over the same parameters used to define $\mathcal{O}$ and interpret $L$) in $L$ with its $K$-induced structure. 
\end{fact}

We can now prove the main result of this section.

\begin{theorem}\label{T:indep-val=ip}
Let $K$ be a t-henselian NIP field, in some language expanding the language of rings. Then $K$ admits at most one definable $V$-topology.
\end{theorem}
\begin{proof}
Assume towards a contradiction that $K$ admits two distinct definable V-topologies. 

We may assume that one of the topologies in question is t-henselian. Indeed, if $K$ is separably closed, then any valuation on $K$ is henselian (and therefore inducing a t-henselian topology). 
If $K$ is a real closed field, then some saturated $\mathcal K\succ K$ has an externally definable henselian valuation (the standard valuation), and $\mathcal K$ expanded with this valuation has NIP. On the other hand, by Lemma \ref{L:def-top}, $\mathcal{K}$ has two distinct definable V-topologies coming from $K$. One of them must be distinct from the topology induced by the externally definable standard valuation. 
Finally, if $K$ is neither real closed nor separably closed, then the $t$-henselian $V$-topology is always 
definable in the language of rings 
(see \cite[p.~203]{Pralg}).

So we are reduced to showing that if $K$ admits two distinct definable V-topologies, one of which is t-henselian, then $K$ has IP. Passing to an elementary extension, by Proposition \ref{P:defV=extdef}, we may assume that $K$ admits two non-trivial externally definable independent valuations, $v_1,v_2$ with $v_1$ t-henselian. Since $K$ has NIP, by Section \ref{ss:NIP} we may add these valuation to the language while preserving NIP. We thus assume that $(K,v_1,v_2)$ has NIP with $v_1$ t-henselian.
%

Let $q=2$ if $\chr[K]\neq 2$ and any other prime if $\chr[K]=2$. 
We may assume that $K$ contains a primitive $q$-th root of unity $\xi_q$. Indeed, by Fact \ref{F:def-of-val-finite-extension}, $(K(\xi_q),v'_1,v'_2)$ is interpretable in $(K,v_1,v_2)$ where $v_i'$ is some extension of $v_i$ to $K(\xi_q)$. Since $v_1,v_2$ were independent, so are $v_1',v_2'$ and since NIP is preserved under interpretations,  it will suffice to show that $(K(\xi_q),v'_1,v'_2)$ has IP. So we may as well assume that $K=K(\xi_q)$.

Since the derivative of the map $x\mapsto x^q$ at $x=1$ is non-zero, by Lemma \ref{L:PZ}, $K^q$ contains an open subset defined by $v_1(x-1)>\gamma$. We may assume that $\gamma>0$ and after replacing $\gamma$ by $q\gamma$ we may further assume that $\gamma$ is $q$-divisible. Note that the assumptions of Lemma \ref{L:correspondence-between-roots} are met. Let $\chi(y):=v_1(y-1)>\gamma\wedge v_2(y-1)>0$. 
\begin{nclaim}
$\chi(y)$ defines an infinite set.
\end{nclaim}
\begin{claimproof}
Since for any $n$ each $v_i$-open set contains $n$ disjoint open balls, the result follows by the strong approximation theorem (Fact \ref{F:strong-approx-ift}).
\end{claimproof}

\begin{nclaim}
If $b \models \chi(y)$ then $b\in K^q$ and for $i=1,2$ there are  one-to-one correspondence \[j_i:\{a\in K: a^q=b\}\to \{0,\dots, q-1\}\] such that $v_1(a-\xi_q^{j_1(a)})>\gamma/q$ and $v_2(a-\xi_q^{j_2(a)})>0$.
\end{nclaim}
\begin{claimproof}
By choice of $\gamma$ if $b\models \chi(y)$, we have $b\in K^q$. The rest is Lemma \ref{L:correspondence-between-roots}.
%
\end{claimproof}

Consider $(\per K,v_1,v_2)$, the perfect hull of $K$, with the unique extensions of $v_1$ and $v_2$ (uniqueness holds by \cite[Corollary 3.2.10]{EnPr}).  In everything that follows, if $K$ is already perfect one may take $p=1$ and the same computations work. Either way, since the trivial valuation also extends uniquely to $\per K$, the valuations $v_1$ and $v_2$ on $\per K$ are still independent. 

Note that $\chi(\per K)\subseteq (\per{K})^q$. For if $v_1(a^{1/p^e}-1)>\gamma$ then $v_1(a-1)>\gamma$ and by the previous claim there exists $b\in K$ satisfying $b^q=a$. Now note that $(b^{1/p^e})^q=a^{1/p^e}$.

Fix a natural number $n$. Let $a_1,\dots,a_n\in \chi(\per K)$ be distinct, and let $C$ be the curve over $\per{K}$ defined by the equations \[y_1^q=a_1+x\]
\[\vdots\]
\[y_n^q=a_n+x.\]

By, e.g. \cite[Chapter 10, Exercise 4]{FrJa}, $C$ is a geometrically integral curve over $K$ (the proof actually shows that it is also of dimension $1$). One easily checks that it is nonsingular by a direct calculation.

Consider the open subset $U_1(\per K)\subseteq \per K$ defined by \[\{v_1(y_i-1)>\gamma/q\}_{i=1}^n\cup\{v_1(x)>\gamma\},\] 
 and for every $S\subseteq \{1,\dots,n\}$ let $U_{2,S}(\per K)\subseteq \per K$ be the open subset defined  by
\[\{v_2(y_i-1)>0\}_{i\in S}\cup \{v_2(y_i-\xi_q)>0\}_{i\notin S}\cup \{v_2(x)>0\}.\]


\begin{nclaim}
$U_1(\per K)\cap C$ is non-empty and for every $S\subseteq \{1,\dots,n\}$, $U_{2,S}(\per K)\cap C$ is non-empty.
\end{nclaim}
\begin{claimproof}
We will show that $U_1(\per K)\cap C$ is infinite, the proof for $U_{2,S}$ is similar. For every $x$ satisfying $v_1(x)>\gamma$, choose $y_i$ to be the $q$-th root of $a_i+x$ satisfying $v_1(y_i-1)>\gamma/q$.
%

\end{claimproof}

%

%
For every natural number $e$, let $\psi_{e,1}(z)$ be the formula $v_1(z-1)>p^e\gamma/q$
 and  $\psi_{e,2}(z)$ be the formula $v_2(z-1)>0$.
Finally, let $\psi_e(u;v)$ be the formula $\forall z\left(z^q=u+v\wedge \psi_{e,1}(z)\rightarrow \psi_{e,2}(z)\right)$.

\begin{nclaim}
There exists a large enough natural number $e$, such that $a_j^{p^e}\in K$ for $1\leq j\leq n$ and for every $S\subseteq \{1,\dots,n\}$ there is an $\tilde{x}_S\in K$ such that 
\[(K,v_1,v_2)\models \psi_e(a_i^{p^e},\tilde{x}_S)\Longleftrightarrow i\in S.\]
%
\end{nclaim}
\begin{claimproof}
For every $S\subseteq\{1,\dots,n\}$, after applying Proposition \ref{P:approx-for-curves} to $U_1(\per K)\cap C$ and $U_{2,S}(\per K)\cap C$, one is supplied with with an $n+1$-tuple \[(x_S,b_{S,1},\dots,b_{S,n})\in \left( U_1(\per K)\cap C\right)\cap \left( U_{2,S}(\per K)\cap C\right).\]

Let $e$ be a large enough integer such that for every such $S$, \[x_S^{p^e},b_{S,1}^{p^e},\dots,b_{S,n}^{p^e},a_1^{p^e},\dots,a_n^{p^e}\in K.\]

Note that for every $S\subseteq \{1,\dots,n\}$ and for every $1\leq i\leq n$, since $(b_{S,i},x_S)$ satisfied 
\[y_i^q=x+a_i\] then, after applying the Frobenius map $e$-times, $(b_{S,i}^{p^e},x_S^{p^e})$ satisfies
\[y_i^q=x+a_i^{p^e}.\]

Consequently, by setting for every $S\subseteq \{1,\dots,n\}$ $\tilde{x}_S:=x_S^{p^e}$, we have
\[(K,v_1,v_2)\models \psi_e(a_i^{p^e},\tilde{x}_S)\Longleftrightarrow i\in S.\]

%
%
\end{claimproof}

Now, let $\psi_1(z,w):=v_1(z-1)>v_1(w)$, $\psi_2(z):=v_2(x-1)>0$ and 
\[
\psi(u;v,w):=(\forall z)(z^q=u+v\land \psi_1(z,w)\to \psi_2(z)).
\] 
We will show that $\psi(u;v,w)$ has IP. We have to show that for all $n$ there are $(c_i: i<n)$ and $\{b_S: S\subseteq [n]\}$ ($b_S$ of length 2) such that $(K,v_1,v_2)\models \psi(c_i;b_S)\iff i\in S$. So given $n$ fix distinct $a_1,\dots a_n\in \chi(\per K)$,  $e$ as provided by the last claim and for all $S\subseteq [n]$ set $b_S:=(\tilde{x}_S, y_S)$ where $y_S$ is an element of $vK$ of valuation $p^e\gamma/q$ and  $\tilde{x}_S$ is provided by the last claim. So $(K,v_1,v_2)\models \psi_e(a_i^{p^e},\tilde{x}_S)$ if and only if $(K,v_1,v_2)\models \psi(a_i^{p^e};b_S)$ which -- by the conclusion of the above claim -- holds if and only if $i\in S$, showing that $\psi(u;v,w)$ has IP. 

This concludes the proof of the theorem.

\end{proof}

Note that a definable valuation on a field induces a definable V-topology.

\begin{corollary}\label{C:hens-incompIP}
Let $(K,v_1,v_2)$ be a field with two non-trivial incomparable (resp. independent) valuations. If $(K,v_1)$ is henselian (resp. t-henselian) then $(K,v_1,v_2)$ has IP.

Consequently, any two externally definable (t-)henselian valuations on a NIP field $K$ (possibly with additional structure) are comparable.
\end{corollary}
\begin{proof}
If $v_1$ and $v_2$ are independent, the result follows from Theorem \ref{T:indep-val=ip}. Otherwise consider $w$, their common coarsening. Assume towards a contradiction that $(K,v_1,v_2)$ has NIP, and thus, since $w$ is definable in the Shelah expansion which also has NIP, $(K,v_1,v_2,w)$ has NIP. Consequently, $(Kw,\overline{v_1},\overline{v_2})$ satisfies the assumptions of Theorem \ref{T:indep-val=ip} and thus has IP, contradiction.

Finally, since the Shelah expansion of $K$ has NIP, adding the two valuations to the structure preserves NIP. If they were not comparable (resp. dependent), that would contradict the above. 
\end{proof}
%
%

\begin{example}\label{E:sep-closed-twoval}
A separably closed field with two incomparable non-trivial valuations has IP.
\end{example}
%
%
%

\begin{corollary}
The following are equivalent:
\begin{enumerate}
\item Every NIP  valued field is henselian.
\item No NIP  field defines two independent valuations.
\item No NIP  field defines two incomparable valuations.
\end{enumerate}
The statement remains true if we replace NIP, throughout, by ``strongly dependent'' or by ``finite dp-rank''.
\end{corollary}
\begin{proof}
We prove the implications for NIP. Obvious adaptations give the implications for finite dp-rank and strongly dependent.

$(1)\implies (2)$. If a NIP field defines two independent valuations, then since by assumption both are henselian it must be separably closed, by \cite[Theorem 4.4.1]{EnPr}. This contradicts Example \ref{E:sep-closed-twoval}.

$(2)\implies (3)$. Let $v_1$ and $v_2$ be two incomparable definable valuations on a NIP field $K$ and let $w$ be their common coarsening. If $w$ is trivial then $v_1$ and $v_2$ are independent and this contradicts $(2)$. If $w$ is not trivial then since $w$ is definable in the Shelah expansion (which has NIP), $(Kw,\bar v_1, \bar v_2)$ the residue field of $w$ together with the two induced valuations, is a NIP field with two independent definable valuations, contradicting $(2)$. 

$(3)\implies (1)$. Let $K$ be a valued field which is not henselian, then there is a finite extension $L/K$ with two incomparable valuations. Since $L$, and both extensions of $v$ to $L$, are interpretable by \cite[Lemma 9.4.8]{johnson}, we get a field with two incomparable valuations which has NIP, contradiction.

\end{proof}

We will say that a field $K$ satisfies the \emph{henselianity conjecture} if for every valuation $v$ on $K$, if $(K,v)$ has NIP then $v$ is henselian. Note that there is no harm in assuming that $K$ has NIP. The above theorem allows us to prove the henselianity conjecture for some valued fields.

\begin{proposition}\label{P:nip-val-rcf}
Let $K$ be a real closed field and $v$ a non-trivial valuation on $K$. Then the following are equivalent:
\begin{enumerate}
\item $(K,v)$ has NIP;
\item $v$ is convex;
\item $v$ is henselian.
\end{enumerate}
\end{proposition}
\begin{proof}

$(1)\implies (2)$. Assume that $(K,v)$ has NIP but $v$ is not convex.  By passing to an elementary extension we may assume that $K$ has a non-trivial convex valuation, e.g. $w_\infty$ whose maximal ideal are the infinitesimals. Since $w_\infty$ is definable in $K^{sh}$, $(K,w_\infty,v)$ has NIP as well. If $v$ and $w_\infty$ were incomparable then, since $(K,w_\infty)$ is henselian by \cite[Theorem 4.3.7]{EnPr}, we would get a contradiction to Corollary \ref{C:hens-incompIP}. Thus they are comparable. If $v$ were coarser than $w_\infty$ then it would be henselian and consequently convex by \cite[Lemma 4.3.6]{EnPr}, contradiction.

As a result, $v$ must be finer than $w_\infty$. Since, as noted above, $w_\infty$ is externally definable, we may pass to the residue field $Kw_\infty$. The residue field $Kw_\infty$ is also real closed and $(Kw_\infty,\bar v)$ has NIP since the induced valuation $\bar v$ is interpretable in $(K,v,w_\infty)$ (which has NIP). But $(Kw_\infty,\bar v)$ has no non-trivial convex valuation, 
so the ordering must be archimedean and the order-topology must be independent from $\bar v$. Consequently $(Kw_\infty,<)$ may be embedded into the reals. Since $Kw_\infty$ is real-closed, and the reals with the ordering are t-henselian, the V-topology the order defines in $Kw_\infty$ is also t-henselian.  Applying Theorem \ref{T:indep-val=ip} we deduce that $(Kw_\infty,\bar v)$ has IP, contradiction. 

$(2)\implies (3)$. This is \cite[Theorem 4.3.7]{EnPr}.

$(3)\implies (1)$. By \cite[Lemma 4.3.6]{EnPr}, $v$ is convex and thus definable in the Shelah expansion $K^{sh}$. Since the latter has NIP, $(K,v)$ has NIP.


\end{proof}

We can generalize the above to any almost real closed field (i.e. a field admitting a henselian valuation with a real closed residue field).

\begin{corollary}\label{C:hen-arc}
Let $K$ be an almost real closed field and $v$ a non-trivial valuation on $K$. Then the following are equivalent:
\begin{enumerate}
\item $(K,v)$ has NIP;
\item $v$ is convex;
\item $v$ is henselian.
\end{enumerate}
\end{corollary}
\begin{proof}
$(1)\implies (3)$. Let $w$ be a henselian valuation on $K$ with real closed residue field. Consequently, by \cite[Theorem A]{JahNIP}, $w$ is definable in $K^{sh}$ and thus $(K,v,w)$ has NIP. So by Corollary \ref{C:hens-incompIP}, $v$ and $w$ must be comparable. If $v$ is coarser than $w$, we are done. Otherwise, since $Kw$ is real closed and $(Kw,\bar v)$ has NIP, where $\bar v$ is the induced valuation on $Kw$, by Proposition \ref{P:nip-val-rcf}, $\bar v$ is henselian. As $v$ is the composition of the henselian valuations $w$ and $\bar v$, it is also henselian.

$(3)\implies (2)$. This is \cite[Lemma 4.3.6]{EnPr}.

$(2)\implies (1)$. Let $w$ be a henselian valuation on $K$ with real closed residue field. By \cite{DelHenselian}, $(K,w)$ has NIP and in particular so does $K$. By \cite[Proposition 2.2(iii)]{delonfarre}, $v$ is henselian and since, by \cite[Lemma 4.3.6]{EnPr}, $Kv$ is not separably closed,  by \cite[Theorem A]{JahNIP}, $v$ is definable in $K^{sh}$ and as a result $(K,v)$ has NIP.
\end{proof}

Once we know the henselianity conjecture for real closed fields, we can now prove:

\begin{proposition}\label{P:def-v-top-is-t-hen}
Let $K$ be a t-henselian NIP field, in some language expanding the language of rings. Then any definable V-topology on $K$ is t-henselian.

Consequently, if $K$ is henselian and $(K,w)$ has NIP then either $w$ is henselian or $w$ is finer than any henselian valuation on $K$. 
\end{proposition}
\begin{proof}
If $K$ is separably closed then any V-topology on $K$ is t-henselian.

Assume that $K$ is real closed. After passing to an elementary extension, by Proposition \ref{P:defV=extdef} it is enough to show that if $(K,v)$ has NIP for some valuation then $v$ is t-henselian. By Proposition \ref{P:nip-val-rcf}, $v$ is henselian.

If $K$ is neither separably closed nor real closed, as in the proof of Theorem \ref{T:indep-val=ip}, $K$ admits a definable $V$-topology inducing the given t-henselian V-topology. By Theorem \ref{T:indep-val=ip}, the given definable V-topology must also be t-henselian. 

Now assume that $(K,v)$ is henselian and $w$ is a non-henselian valuation on $K$ such that $(K,w)$ has NIP. By the above, $v$ and $w$ are dependent. Assume that they are incomparable and let $u$ be their join. Since $u$ is externally definable in $(K,w)$, $(Ku,\bar w)$ has NIP. Note that $\bar w$ is not henselian, but $Ku$ is henselian ($\bar v$ is henselian), so by the above $\bar v$ and $\bar w$ must be dependent. Contradicting the fact that $v$ and $w$ are incomparable.
\end{proof}

In particular, if $K$ is henselian and $(K,v)$ has NIP but not henselian then all henselian valuations are externally definable in $(K,v)$.

The above corollary allows us to prove the following:

\begin{corollary}\label{C:hen-conj-goes-up}
Let $K$ be a NIP field admitting a henselian valuation $v$ such that $Kv$ satisfies the henselianity conjecture. Then $K$ satisfies the henselianity conjecture.
\end{corollary}
\begin{proof}
Let $w$ be a valuation on $K$ with $(K,w)$ NIP. If $w$ is not henselian then by Proposition \ref{P:def-v-top-is-t-hen}, $w$ must be finer than $v$. Since $v$ is then externally definable in $(K,w)$, $(Kv,\bar w)$ has NIP. Using the fact that $Kv$ satisfies the henselianity conjecture, $\bar w$ is henselian and consequently $w$ is henselian as well.
\end{proof}

Note that this implies that any Hahn field over a field satisfying the henselianity conjecture, satisfies the henselianity conjecture. 


In the following we will show that every infinite dp-minimal field satisfies the henselianity conjecture for NIP fields and thus giving a wide class of examples satisfying the statement of the previous proposition.
\begin{corollary}\label{C:hen-con-dpmin-res}
Every dp-minimal field satisfies the henselianity conjecture for NIP fields.

Consequently, if $K$ is a NIP field admitting a henselian valuation with dp-minimal residue field then $K$ satisfies the henselianity conjecture.
\end{corollary}
\begin{proof}
Let $K$ be a dp-minimal field. We first note that $K$ admits a henselian (possibly trivial) valuation with a residue field which is either algebraically closed, real closed or finite. The proof of this fact is a direct consequence of a variant of \cite[Theorem 9.7.2]{johnson} which does not use the saturation requirement (see \cite[Remark 3.2]{HaHaJa}). Alternatively, one may use Proposition \ref{P:jahnke}(1) and the fact that Shelah's conjecture for dp-minimal fields is true.

If $v$ is trivial then we are done, otherwise we may use Corollary \ref{C:hen-conj-goes-up}.
%
\end{proof}

\section{Shelah's Conjecture Implies Henselianity Conjecture}
Shelah's Conjecture on NIP fields states that
\begin{conjecture}\label{Shelah}
	Any infinite NIP field which is neither real closed nor separably closed admits a non-trivial henselian valuation. 
\end{conjecture}

\begin{remark}
There are corresponding conjectures for fields of finite dp-rank and for strongly dependent fields, just replace separably closed by algebraically closed.
\end{remark}

\begin{proposition}\label{P:jahnke}
Shelah's conjecture implies the following:
\begin{enumerate}
\item Any infinite NIP field $K$ admits a henselian valuation $v$ (possibly trivial) such that $Kv$ is real closed, separably closed or finite. 
\item Any infinite NIP field which is neither real closed nor separably closed admits a non-trivial definable henselian valuation. 
\item Any NIP non-trivially valued field admits a non-trivial definable henselian valuation.
\end{enumerate}
\end{proposition}
\begin{remark}
Similar proofs give the same results for strongly dependent fields and fields of finite dp-rank.
\end{remark}
\begin{proof}
\begin{enumerate}
\item We may assume that $K$ is neither separably closed nor real closed (otherwise take $v$ to be the trivial valuation). We may assume that every non-trivial henselian valuation on $K$ has non separably closed residue field.
	By \cite[Theorem 4.4.2]{EnPr}, there exists a finest henselian valuation with non separably closed residue field, the canonical valuation $v_K$. It is non trivial since $K$ admits a non-trivial henselian valuation by Shelah's conjecture. Notice that $Kv_K$ is non henselian. Indeed, if it were henselian then composing the corresponding place with $K\to Kv_K$ would yield a henselian valuation on $K$ strictly finer than $v_K$, contradiction.
	 Since $Kv_K$ is not separably closed, by \cite[Theorem A]{JahNIP} $v_K$ is externally definable in $K$. As a result, $(K,v_K)$ has NIP and so does $Kv_{K}$. Applying Conjecture \ref{Shelah} again, we get that it is finite or real closed. 
	 
\item Let $K$ be an infinite NIP field which is not real closed nor separably closed. Using (1), we are supplied with a, necessarily non-trivial, henselian valuation $v$ such that $Kv$ is real closed, separably closed or finite. If $Kv$ is real closed or separably closed then, by \cite[Theorem 3.10, Corollary 3.11]{JahKoe2015}, $K$ admits a non-trivial definable henselian valuation. If $Kv$ is finite then $v$ is definable in $K$ by \cite[Example after Proposition 3.1]{JahKoe2015}.
\item Let $(K,v)$ be a NIP valued field and assume that $v$ is non-trivial. If $K$ is separably closed then $v$ is henselian, if it is real closed $v$ is henselian by Proposition \ref{P:nip-val-rcf}. Since $K$ can not be finite ($v$ is non-trivial), by (2), either way $(K,v)$ admits a non-trivial definable henselian valuation. 
\end{enumerate}
\end{proof}

The following is a corollary of Proposition \ref{P:jahnke} and Corollarys \ref{C:hen-conj-goes-up}. We give a short alternative proof. 

\begin{proposition}\label{P:shelah-implies-hens}
Shelah's conjecture implies that if $(K,v)$ is a NIP non-trivially valued field then $(K,v)$ is henselian. Shelah's conjecture for strongly dependent fields (fields of finite dp-rank) implies the henselianity conjecture for strongly dependent fields (fields of finite dp-rank). 
\end{proposition}
\begin{proof}
We prove the proposition for NIP valued field. A similar proof gives the finite dp-rank and strongly dependent cases.


We first observe that by Proposition \ref{P:jahnke}, $K$ admits a non-trivial definable henselian valuation.
Let $\mathcal{O}_w$ the intersection of all externally definable henselian valuation rings in $(K,v)$. Note that it is a valuation ring and henselian since any two of the intersectants are comparable, by Corollary \ref{C:hens-incompIP}. Denote by $w$ the valuation associated with $\mathcal{O}_w$.

\begin{claim}
$w$ and $v$ are comparable.
\end{claim}
\begin{proof}
Assume that they are incomparable. Denote by $\tilde{w}$, with valuation ring $\mathcal{O}_{\tilde{w}}$, their common coarsening. It is externally definable in $(K,v)$ and henselian. 

If there were no externally definable valuation rings between $\mathcal{O}_w$ and $\mathcal{O}_{\tilde{w}}$ it would mean that $\mathcal{O}_w$ is externally definable. Thus $(K,v,w)$ has NIP, contradicting Corollary \ref{C:hens-incompIP}.

Otherwise, let $\mathcal{O}_{w'}$ be an externally definable henselian valuation ring between the two of them. Consequently, $(K,v,w')$ has NIP and $w'$, $v$ are incomparable contradicting Corollary \ref{C:hens-incompIP}.
\end{proof}

\begin{claim}
There is no definable henselian valuation on $Kw$.
\end{claim}
\begin{claimproof}
If $Kw$ is non-separably closed then by \cite[Proposition 2.4]{JahNIP}, $w$ is externally definable in $(K,v)$. If there were a definable henselian valuation on $Kw$, the composition with $w$ would be an externally definable henselian valuation strictly finer than $w$, contradiction.

If $Kw$ is separably closed, there can be no definable valuation since $Kw$ is stable.
\end{claimproof}

If $\mathcal{O}_w\subseteq \mathcal{O}_v$ we are done since then $v$ is 
henselian. We will thus assume that $\mathcal{O}_v\subsetneq \mathcal{O}_w$, so $\mathcal{O}_w$ is externally definable and $(K,v,w)$ has NIP.

Consequently, $(Kw,\bar v)$ has NIP and $Kw$ admits no non-trivial definable henselian valuation. By Proposition \ref{P:jahnke}(2), $Kw$ must be separably closed or real closed. Either way, $\bar v$ must be henselian, where the real closed case follows from Proposition \ref{P:nip-val-rcf}.
\end{proof}

Recently Krapp-Kuhlmann-Leh\'ericy \cite{KKL} showed, using different methods, that Shelah's conjecture restricted to ordered strongly dependent fields is equivalent to the fact that every such field is almost real closed. Consequently, Shelah's conjecture for strongly dependent ordered fields implies, using Corollary \ref{C:hen-arc}, the henselianity conjecture for such fields.

\section{Non-Empty Interior and Henselianity}\label{S: NEIH}
Several theorems of the sort "model theoretic dividing line implies henselianity of valued fields'' (e.g., \cite{MMvdD} \cite{weaklyrcf} \cite{Guingona2014} \cite{johnson} and \cite{JSW})  use the model theoretic assumptions to deduce that every Zariski dense definable set has non-empty interior in the valuation topology. Recently, and independently from us, it was observed by Sinclair \cite{Peter} that for valued fields of finite dp-rank henselianity of the valued field $(K,v)$ is equivalent to: for every coarsening $v_\Delta$ of $v$, every infinite definable subset of $Kv_\Delta$ in the language $(Kv_\Delta,\bar v)$ has non empty interior with respect to the $\bar v$-adic topology.

The aim of this section is to expand this observation to valued fields with no model theoretic constraint, and give in the process a new (to the best of our knowledge) characterisation of henselianity. 

As the following example shows,  it is not true that in an infinite non-perfect NIP field an infinite definable set must have non empty interior. Consequently, the above topological condition must be adjusted. This will be achieved by using t-henselianity.

\begin{example}
Let $(K,v)$ be an infinite NIP imperfect valued field of characteristic $p$. 
Since $K^p$ is a definable subfield it is NIP and hence Artin-Schreier closed so by \cite[Theorem 1.11]{FVK-densepefect} if $K\setminus K^p$ had non-empty interior $K^p$ would intersect it, which is absurd.
\end{example}

The basic outline of the following proofs is well known. Our exposition owes much to \cite[Section 4]{JSW}. We use the following fact, see \cite[Lemma 3.9]{Guingona2014}, for instance, for a proof:
\begin{fact}\label{F:well-used}
Let $F$ be a field extension of $K$ and let $\alpha\in F\setminus K$ be algebraic over $K$. Let $\alpha=\alpha_1,\dots,\alpha_n$ be the conjugates of $\alpha$ over $K$ and let $g$ be given by:
\[g(X_1,\dots,X_{n-1},Y):=\prod_{i=1}^n\left(Y-\sum_{j=0}^{n-1}\alpha^j_iX_j\right).\]
Then $g\in K[X_0,\dots,X_{n-1},Y]$ and there are $G_0,\dots,G_{n-1}\in K[X_0,\dots,X_{n-1}]$ such that
\[g(X_0,\dots,X_{n-1},Y)=\sum_{j=0}^{n-1}G_j(X_0,\dots,X_{n-1})Y^j+Y^n.\]
Letting $G=(G_0,\dots,G_{n-1})$ we have:
\begin{enumerate}
\item If $\bar c=(c_0,\dots,c_{n-1})\in K^n$ and $c_j\neq 0$ for some $j$, then $g(\bar c,Y)$ has no roots in $K$;
\item There is a $\bar d=(d_0,\dots,d_{n-1})\in K^n$ such that $d_j\neq 0$ for some $j$ and $|J_G(\bar d)|\neq 0$.
\end{enumerate}
\end{fact}

\begin{proposition}\label{P:some-approx}
Suppose $(K,v)$ is a valued field, and let $(L,v^h)$ be a henselian immediate separable algebraic extension of $(K,v)$. Furthermore, suppose that for every polynomial map $P:K^n\to K^n$ and $v$-open subset $U\subseteq K^n$ such that $|J_P(U)|\neq \{0\} $ the image $P(U)$ has non-empty interior.

Let $a\in L$ such that for any $\gamma\in v^hL=vK$ there is some $b\in K$ such that $v^h(b-a)\geq \gamma$. Then $a\in K$.
\end{proposition}
\begin{proof}
Suppose $a$ has degree $n$ over $K$ and let $g$ and $G$ be as in Fact \ref{F:well-used}. Let $\bar d$ be as in $(2)$ in Fact \ref{F:well-used}. By applying Lemma \ref{L:PZ} to $(K^h,v^h)$, the henselization of $(K,v)$, there is a $v$-open subset  $U^h\subseteq (K^h)^n$ containing $\bar d$ such that the restriction of $G$ (viewed as function in $K^h$) to $U^h$ is a homeomorphism. Thus the restriction of $G$ (as a function in $K$) to $U:=U^h\cap K^n$ is injective.  By assumption, $G(U)$ has non empty interior. 
We continue the proof exactly as in \cite[Proposition 4.4]{JSW}, and the result is that $a\in K$.

\end{proof}

To prove the main result of this section we will need the following variant of Taylor's theorem:

\begin{fact}\cite[Lemma 24.59]{FVK}\label{F:taylor}
Let $(K,v)$ be a valued field with valuation ring $\mathcal{O}_v$ and let $f\in \mathcal{O}_v[X]$ be a polynomial. Then there exists $H_f(X,Y)\in \mathcal{O}_v[X,Y]$ such that 
\[f(Y)-f(X)=f'(X)(Y-X)+(Y-X)^2H_f(X,Y).\]
\end{fact}

\begin{theorem}\label{T:t-h}
Let $(K,v)$ be a valued field. The following are equivalent:
\begin{enumerate}
\item $(K,v)$ is henselian;
\item $(Kv_\Delta,\bar v)$ is t-henselian for every coarsening $v_\Delta$ of $v$.

\item For every coarsening $v_\Delta$ of $v$, polynomial map $P: (Kv_\Delta)^n\to (Kv_\Delta)^n$ and $a\in (Kv_\Delta)^n$ such that the Jacobian determinant is non zero at $a$ and $a\in U\subseteq (Kv_\Delta)^n$ where $U$ is a $\bar{v}$-open set, $P(U)$ has non-empty interior.
\item For every coarsening $v_\Delta$ of $v$, polynomial map $P: (Kv_\Delta)^n\to (Kv_\Delta)^n$ with Jacobian determinant not identically zero and $U\subseteq (Kv_\Delta)^n$ where $U$ is a $\bar v$-open set, $P(U)$ has non-empty interior.

\end{enumerate}
\end{theorem}
\begin{proof}
$(1)\implies (2)$. For every coarsening $v_\Delta$ of $v$, $(Kv_\Delta,\bar v)$ is henselian and in particular t-henselian.

$(2)\implies (3)$. Such a point $a$ is an \'etale point of $P$ so the result follows by Proposition \ref{P:etale=non-empty interior}.

$(3)\implies (4)$. Let $P: (Kv_\Delta)^n\to (Kv_\Delta)^n$ be a polynomial map with non-vanishing Jacobian determinant and $U\subseteq K^n$ a $\bar{v}$-open subset. Since the set of points of $K^n$ where the Jacobian determinant in non-zero is open, and $U$ is Zariski dense, there exists such a point in $U$.

$(4)\implies (1)$. We denote by $\mathcal{O}_v$ and $vK$ the valuation ring and value group of $v$, respectively. We suppose towards a contradiction that $v$ is not henselian. By \cite[Theorem 9.1(2)]{FVK}, there is a polynomial \[p(X)=X^n+c_{n-1}X^{n-1}+\sum_{i=0}^{n-2}c_iX^i\in \mathcal{O}_v[X]\] 
with $v(c_{n-1})=0, v(c_i)>0$ for all $i<n-1$, such that $p$ has no root in $K$. Let $a\in K^h$ be such that $p(a)=0$, $v^h(a-c_{n-1})>0$ and $v^h(p'(a))=0$, indeed, $0$ is the only multiple root of $\bar p(X)=X^{n-1}(X+\bar c_{n-1})$ (see  \cite[Theorem 9.1(2)]{FVK}). Consider the subset \[S:=\{v^h(b-a)\in vK:b\in K, v^h(b-a)>0\}\] of $vK$, and let $\Delta$ be the convex subgroup of $vK$ generated by $S$. Note that $\Delta\neq vK$, for otherwise, our assumption allows us to apply Proposition \ref{P:some-approx} to deduce that $a\in K$.

\begin{claim}
$S$ is cofinal in $\Delta$.
\end{claim}
\begin{claimproof}
The proof is basically the same as in \cite[Claim 5.12.1]{weaklyrcf}. The main difference is that instead of using Taylor's approximation theorem (which does not hold for positive characteristic) we apply Fact \ref{F:taylor}.

It suffices to prove that if $\gamma\in S$  (so $\gamma>0$) then there is some $c\in K$ satisfying $v^h(c-a)\geq 2\gamma$.  Let $\gamma\in S$ and let $b\in K$ such that $v^h(b-a)=\gamma$. Note that, by using Fact \ref{F:taylor} twice, there are $H_1,H_2\in K^h$ with $v^h(H_1),v^h(H_2)\geq 0$, such that
\[p(b)=p'(a)(b-a)+(b-a)^2H_1 \text{ and }\]
\[-p(b)=p'(b)(a-b)+(a-b)^2H_2.\]
As a result, $v^h(p'(b))=v^h(p'(a)+(b-a)(H_1+H_2))$ so, since $v^h(p'(a))=0$ and $v^h((b-a)(H_1+H_2))>0$, $v^h(p'(b))=v^h(p'(a))=0$. In particular $p'(b)\neq 0$. Using Fact \ref{F:taylor}, again, 
\[p(a)-p(b)=(a-b)p'(b)+(a-b)^2H,\] where $v^h(H)\geq 0$, and setting $c:=b-p(b)/p'(b)\in K$ we get that
\[c-a=(a-b)^2\frac{H}{p'(b)}.\] 

Consequently, since $v^h(p'(b))=0$, \[v^h(c-a)=2v^h((a-b)H)\geq 2\gamma.\]
\end{claimproof}

Let $v_\Delta$ be the coarsening of $v$ with valued group $\Gamma_v/\Delta$ and let $v^h_\Delta$ be the corresponding coarsening of $v^h$.  We let $\bar v$ be the non-trivial valuation induced on $Kv_\Delta$ by $v$.

\begin{claim}
There exists $c\in K$, with $v_\Delta(c-a)>0$,  such that $p(c)\in \mathcal{M}_{v_\Delta}$.
\end{claim}
\begin{claimproof}
By the definition of $\Delta$ and the first claim, the residue $a v^h_\Delta$ is approximated arbitrarily well in the residue field $Kv_\Delta$ (with respect to the valuation $\bar v$). Note that $(K^hv^h_\Delta, \bar{v^h})$ is a henselian immediate separable algebraic extension of $(Kv_\Delta,\bar v)$. 

Thus, by $(4)$,  we may apply Proposition \ref{P:some-approx} to $(Kv_\Delta,\bar v)$, to deduce that $a v^h_{\Delta}\in Kv_\Delta$. Take some $c\in K$ with the same residue (with respect to $v_\Delta$) as $a$. In particular, $cv_\Delta$ is a root of the polynomial $\bar p(X)$ (that is $p(X)$ considered in $Kv_\Delta$), as needed.
\end{claimproof}

We declare \[J:=\{b\in K: v(b-a)>0\}.\]
Then, as $c-a\in \mathcal{M}_{v_\Delta}\subseteq \mathcal{M}_v$, we have $c\in J$.
\begin{claim}
For all $b\in J$, $v^h(b-a)=v(p(b))$.
\end{claim}
\begin{claimproof}
Again, by Fact \ref{F:taylor}
\[p(b)=p'(a)(b-a)+(b-a)^2H,\] with $v^h(H)\geq 0$. Since $v^h(p'(a))=0$, and $0<v^h(b-a)<v^h((b-a)^2H)$, the assertion follows.
\end{claimproof}
By the definition of $\Delta$, $v_\Delta(p(b))=0$ for any $b\in J$. This contradicts $p(c)\in \mathcal{M}_{v_\Delta}$, and this finishes the proof.

\end{proof}

Using this result we may give a reformulation of Shelah's Conjecture.
\begin{corollary}\label{C:reform-shelah}
The following are equivalent:
\begin{enumerate}
\item Every infinite NIP field is either separably closed, real closed or admits a non-trivial henselian valuation.
\item Every infinite NIP field is either separably closed, real closed or admits a non-trivial definable henselian valuation.
\item Every infinite NIP field is either t-henselian or the algebraic closure of a finite field.
\item Every infinite NIP field is elementary equivalent to a henselian field.
\end{enumerate}
\end{corollary}
\begin{proof}
$(2)\implies (1)$. Straightforward.

$(1)\implies (4)$. Separably closed fields and real closed fields are elementary equivalent to henselian fields.

$(4)\implies (3)$. If $K$ is not the algebraic closure of a finite field it must admit a V-topology $\tau$. If $K$ is separably closed this topology is necessarily t-henselian so assume this is not the case. By \cite[Remark 7.11]{PrZ1978} and assumption $(4)$, $K$ admits a t-henselian V-topology.

$(3)\implies (1)$. Let $K$ be an infinite field which is not separably closed nor real closed. Since, by $(3)$, $K$ is t-henselian, by either \cite[Theorem 5.2]{JahKoe} or \cite[Theorem 6.19]{dupont}, $K$ admits a definable non-trivial t-henselian valuation $v$. We will show that $v$ is henselian using Theorem \ref{T:t-h}. Let $w$ be any proper coarsening of $v$. Since $(K,v)$ has NIP, so does $(Kw,\bar v)$ and since $w$ is a proper coarsening $Kw$ is infinite. By assumption $(3)$, $Kw$ is t-henselian. By Proposition \ref{P:def-v-top-is-t-hen}, $\bar v$ is t-henselian, as needed.

$(1)\implies (2)$. Proposition \ref{P:jahnke}(3).
\end{proof}

Since comparable valuations induce the same topology, if a valuation has a non-trivial henselian coarsening it is t-henselian. The other direction is false, e.g. in \cite[page 338]{PrZ1978} Prestel and Ziegler construct a t-henselian field (i.e. elementary equivalent to a field admitting a henselian valuation) of characteristic 0, which is neither real closed, algebraically closed nor does it admit a henselian valuation. Now, since the field is neither real closed nor separably closed it admits a definable valuation inducing the t-henselian topology by \cite[Theorem 5.2]{JahKoe} and \cite[Theorem 6.19]{dupont}. Since $K$ is not henselian this valuation can not have a henselian coarsening. However, in the spirit of Theorem \ref{T:t-h}, we prove the following related result:

\begin{proposition}\label{P:t-hen-algext}
Let $(K,v)$ be a valued field with $K$ not real closed. Then the following are 
equivalent:
\begin{enumerate}
\item $v$ has a non-trivial henselian coarsening;
\item For every algebraic extension $L/K$, any extension of $v$ to $L$ 
is t-henselian.
\end{enumerate}
\end{proposition}

Before we proceed with the proof, we recall that a field $K$ is called \emph{Euclidean} if  $[K(2):K]=2$. Any Euclidean field is uniquely ordered, the positive elements being exactly the squares.

\begin{proof}
If $K$ is separably closed then any valuation on $K$ is henselian. We 
may thus assume that $K$ is neither real closed nor separably closed.

$(1)\implies (2)$. Let $w$ denote the henselian coarsening of $v$, $L/K$ 
be an algebraic extension and let $v'$ be any extension of $v$ to $L$. 
By \cite[Lemma 3.1.5]{EnPr} the unique extension $w'$ of $w$ to $L$
coarsens $v'$. Since $w$ is henselian, so is $w'$ and since $w'$ is a 
coarsening of $v'$, the latter is t-henselian.

$(2)\implies (1)$. We plan to use \cite[Proposition 3.1(a)]{Koe99}. For 
that, we will need some basic theory of the absolute Galois group, see e.g. 
\cite[Chapter 1, Sections 22.8 and 22.9]{FrJa}. Let $p$ be a prime 
number dividing $\# G_K$, the order of the absolute Galois group of $K$, 
and in case $K$ is Euclidean choose $p\neq 2$. Note that if $K$ is 
Euclidean and $2$ is the only prime number dividing $\# G_K$ then $K$ is 
necessarily real closed, contradicting our assumption. 

Let $L$ be the fixed field of some $p$-Sylow subgroup of $G_K$. Let $v'$ be any 
extension of $v$ to $L$. By assumption, $(L,v')$ is t-henselian. Note 
that since $L$ admits no Galois extensions of degree prime to $p$, L 
contains a primitive $p$th root of unity in case $\chr[L]\neq p$. Also, 
by construction, $L$ is not Euclidean in case $p=2$. Thus, as $L$ is elementary equivalent to a henselian field (and thus also to a p-henselian field, i.e. admitting a non-trivial $p$-henselian valuation) by \cite[Theorem 7.2]{PrZ1978}, by \cite[Corollary 2.2]{Koenigsmann}), $L$ 
is $p$-henselian. Note that since $L(p)=L^{sep}$, by choice of $L$, $p$-henselianity and henselianity (and hence also $t$-henselianity) coincide on $L$.
 
Consider the henselian valuation $u$ on $L$ which is defined as follows:
if either $p\neq 2$ or $p=2$ and $L$ admits no henselian valuation with real closed residue field, we take $u$ to
 be $v_L$, the canonical henselian valuation on $L$. Otherwise, let $u$ be the 
coarsest henselian valuation with real closed residue field 
(note that the latter exists and non-trivial by \cite[Proposition 2.1(iv)]{delonfarre} and since $L$ is not Euclidean and hence in particular not real closed).
Either way, $u$ coarsens $v_L$.

In particular, $u$ and $v'$ are both t-henselian (recall that $v'$ is 
t-henselian by assumption). Thus, they are dependent since otherwise $L$ 
would be separably closed by \cite[Theorem 7.9]{EnPr}. As a result, they 
have a common non-trivial henselian coarsening $w$.

In order to apply \cite[Proposition 3.1(a)]{Koe99} we must check that:
\begin{enumerate}
\item $w$ is a coarsening of the canonical henselian valuation $v_L$.
\item If $p=2$ and $Lw$ is real closed then $w$ has no proper coarsenings with real closed residue field.
\end{enumerate}

Since both these properties hold for $u$, they trivially follow for its 
coarsening $w$. Thus, by  \cite[Proposition 3.1(a)]{Koe99}, the 
restriction of $w$ to $K$ is a henselian valuation, it is non-trivial 
since $w$ is non-trivial. Finally, since $w$ coarsens $v'$, the restriction of $w$ to $K$ coarsens $v$.
\end{proof}

\bibliographystyle{plain}
\bibliography{nip-valued-fields}

\end{document}